\documentclass[12pt]{amsart}
\usepackage{ amsmath, amsthm, amsfonts, amssymb, color, xcolor}
 \usepackage{mathrsfs}
\usepackage{amsfonts, amsmath}
 \usepackage{amsmath,amstext,amsthm,amssymb,amsxtra,verbatim}
 \usepackage{multicol}
 \allowdisplaybreaks
\begin{document}
\baselineskip 18pt
\hfuzz=6pt

\newtheorem{theorem}{Theorem}[section]
\newtheorem{prop}[theorem]{Proposition}
\newtheorem{lemma}[theorem]{Lemma}
\newtheorem{definition}[theorem]{Definition}
\newtheorem{cor}[theorem]{Corollary}
\newtheorem{example}[theorem]{Example}
\newtheorem{remark}[theorem]{Remark}
\newcommand{\ra}{\rightarrow}
\newcommand{\ccc}{{\mathcal C}}
\newcommand{\one}{1\hspace{-4.5pt}1}

\newcommand{\wh}{\widehat}
\newcommand{\f}{\frac}
\newcommand{\df}{\dfrac}
\newcommand{\sgn}{\textup{sgn\,}}
 \newcommand{\rn}{\mathbb R^n}
  \newcommand{\si}{\sigma}
  \newcommand{\ga}{\gamma}
   \newcommand{\nf}{\infty}
\newcommand{\p}{\partial}
\newcommand{\De}{\Delta}

%%Some Commands used by Hanh%%%%%%%%%
\newcommand{\norm}[1]{\left\|{#1}\right\|}%
\newcommand{\supp}{\operatorname{supp}}

\newcommand{\tf}{\tfrac}
\newcommand{\qq}{\quad\quad}
\newcommand{\lab}{\label}
\newcommand{\zzz}{\mathbf Z}
\newcommand{\li}{L^{\infty}}
\newcommand{\intrn}{\int_{\rn}}
\newcommand{\qqq}{\quad\quad\quad}
\numberwithin{equation}{section}

%%End commands inputed by Hanh%%%%%%%

%% Commands from Danqing
\newcommand{\vp}{\varphi}
\newcommand{\al}{\alpha}
\newcommand{\R}{\RR}
\newcommand{\intr}{\int_{\R}}
\newcommand{\intrr}{\int_{\R^2}}
\newcommand{\de}{\delta}
\newcommand{\om}{\omega}
\newcommand{\Tht}{\Theta}
\newcommand{\tht}{\theta}
%\newcounter{Question}
\newcounter{question}
\newcommand{\qt}{%
        \stepcounter{question}%
        \thequestion}
\newcommand{\bq}{\fbox{Q\qt}\ }
%box of questions. The last "\ " after \qt is used to make a space between the box the word
%after it.
\renewcommand{\wp}{\Psi}
\newcommand{\sgg}{\si_{gg}}
\newcommand{\sbtj}{\si_{b2j}}
\newcommand{\sbrk}{\si_{b3k}}
\newcommand{\sbzg}{\si_{b_0g}}

\newcommand{\abs}[1]{\left\vert #1\right\vert}%

\newcommand{\be}{\beta}%

\def\RR{\mathbb R}
\def\bbr{\mathbb R}

\title[Multilinear Multiplier Theorems]
{Multilinear Multiplier Theorems and Applications}
%and Applications to Commutators }
% Based on GraHeNguYan20150920.tex

\author{Loukas Grafakos}

\address{Department of Mathematics, University of Missouri, Columbia MO 65211, USA}
\email{grafakosl@missouri.edu}

\author{Danqing He}

\address{Department of Mathematics, Sun Yat-sen  University, Guangzhou, 510275, P. R. China}
\email{hedanqing@mail.sysu.edu.cn}

\author{Hanh Van Nguyen}

\address{Department of Mathematics, University of Alabama, Tuscaloosa, AL 35487, USA}
\email{hvnguyen@ua.edu}

\author{Lixin Yan}

\address{Department of Mathematics, Sun Yat-sen  University, Guangzhou, 510275, P. R. China}
\email{mcsylx@mail.sysu.edu.cn}

\thanks{{\it Keywords and phases:} Multilinear   operators, multiplier operators,   Calder\'on commutators.}
\thanks{The first three authors would like to acknowledge the support of Simons Foundation.
The fourth author was supported by
 NNSF of China (Grant No. 11371378, 11471338 and 11521101).}

\begin{abstract}
We obtain new  multilinear multiplier theorems for  symbols  of restricted smoothness  which lie locally in certain Sobolev spaces.  We provide applications concerning the boundedness of the commutators of Calder\'on and Calder\'on-Coifman-Journ\'e. 
\end{abstract}

\maketitle

%\tableofcontents

\bigskip

 \section{Introduction}
\setcounter{equation}{0}

The  theory of multilinear multipliers has made significant advances in recent years. An $n$-dimensional
$m$-linear multiplier  is a bounded function $\si$ on $(\rn)^m $ associated with an $m$-linear
operator
$T_\si$ on $\rn\times \cdots\times \rn$ in the following way:
\begin{equation}\label{Ts}
T_\si(f_1,\dots , f_m)(x) =
\int_{(\rn)^m } \!\!\!\wh{f_1}(\xi_1) \cdots \wh{f_m}(\xi_m) \si(\xi_1,\dots , \xi_m) e^{2\pi i x\cdot (\xi_1+\cdots +\xi_m)}
d\xi_1\!\cdots \!d\xi_m,
\end{equation}
where $f_j$, $j=1,\dots , m$, are Schwartz functions in $\rn$, and
$\wh{f_j}(\xi_j) = \int_{\rn} f_j(x) e^{-2\pi i x\cdot \xi_j} dx$ is the
Fourier transform of $f_j$.  A classical result of Coifman and Meyer \cite{CM-G, CM2} says that  if
for all sufficiently large multiindices $\al_1,\dots , \al_m\in (\mathbb Z^+\cup \{0\})^n$ we have
\begin{equation}\label{mi}
\Big|  \p_{ \xi_1}^{\al_1}  \cdots  \p_{ \xi_m}^{\al_m}  \si (\xi_1,\dots , \xi_m) \Big| \lesssim
%C_{\al_1,\dots , \al_m}
(|\xi_1|+\cdots +|\xi_m|)^{-(|\al_1|+\cdots + |\al_m|)}
\end{equation}
for all $(\xi_1,\dots , \xi_m)  \in (\rn)^m \setminus \{(0,\dots ,0)\}$, then $T_\si$ admits a bounded extension from
$L^{p_1}(\rn) \times \cdots \times L^{p_m}(\rn)$ to $L^p(\rn)$
when $1<p_1,\dots , p_m\le \nf$, $1/p=1/p_1+\cdots +1/p_m$, and
$1\le p<\nf$. The extension of this theorem to indices $p>1/m$ was simultaneously obtained by
Kenig and Stein \cite{KS} (when $m=2$) and Grafakos and Torres \cite{GT2}.
 This theorem provides an $m$-linear extension of Mikhlin's classical
linear multiplier result \cite{Mi}. H\"ormander \cite{Ho} obtained an improvement of Mikhlin's theorem
showing that when $m=1$, $T_\si$ maps
$L^{p_1}(\rn)$ to $L^{p_1}(\rn)$, $1<{p_1}<\nf$    under the       weaker condition
\begin{equation}\label{ho}
\sup_{j\in \mathbb Z}\big\| (I-\De)^{s/2} \big(\si(2^j \cdot ) \wh{\Psi}\big) \big\|_{L^2(\rn)} <\nf,
\end{equation}
where $s>n/2$ and $\wh{\Psi}$ is a smooth function   supported in an annulus centered at the origin.
Here $\De$ is the Laplacian and
$(I-\De)^{s/2}$ is an operator given on the Fourier transform side by multiplication with $(1+4\pi^2 |\xi|^2)^{s/2}$.
H\"ormander's theorem was  extended to $L^r$-based Sobolev spaces and to indices ${p_1}\le 1$, with
$L^{p_1}$ replaced by the Hardy space $H^{p_1}$, by Calder\'on and Torchinsky \cite{CT}.

The adaptation of    H\"ormander's theorem   to the multilinear setting was first obtained by Tomita \cite{To}. This
theorem was later extended by Grafakos and Si \cite{GrSi} to the range $p<1$ by replacing $L^2$-based Sobolev spaces
by $L^r$-based Sobolev spaces. The endpoint cases where some $p_j$ are equal to infinity were treated by
Grafakos, Miyachi, and Tomita \cite{GrMiTo}.
Fujita and Tomita \cite{FuTo} provided weighted  extensions of these results and also noticed that the operator
$(I-\De)^{s/2}$ in $(\rn)^m$ can be replaced by 
$(I-\De_{\xi_1})^{s_1/2}\cdots (I-\De_{\xi_m})^{s_m/2}$, where $\De_{\xi_j}$ is
the Laplacian in the $\xi_j$th variable.
The bilinear version of the Calder\'on and Torchinsky theorem
was proved by Miyachi and Tomita \cite{MiTo}, while the $m$-linear version (for general $m$) was  proved by
Grafakos and Nguyen \cite{GrNg} and Grafakos, Miyachi, Nguyen, and Tomita \cite{GrMiNgTo}.

To study certain   multilinear singular integrals, such as multicommutators,  
there is a need for a multilinear multiplier theorem  that can handle  symbols on $(\rn)^m$ which, for instance, have
one derivative in each variable  but no   two derivatives in a given variable.
We notice that  in the case where   $s_{j} $ are positive integers for all $j $, replacing
$(I-\De)^{s/2}$  on $(\rn)^m$ by $(I-\De_{\xi_1})^{s_1/2}\cdots (I-\De_{\xi_m})^{s_m/2}$,  as in Fujita and Tomita \cite{FuTo},
   reflects the following decay condition for the derivatives of $\si$
\begin{equation}\label{GG1}
\big|\p_{\xi_{1}}^{\be_{1 }}
\p_{\xi_{2 }}^{\be_{2 }}
\cdots \p_{\xi_{m }}^{\be_{m }}  \si (\xi_1,\dots , \xi_m) \big| \lesssim (|\xi_1|+\cdots +|\xi_m|)^{-\sum_{j=1}^m
 |  \be_{j }|} ,
\end{equation}
where each multiindex $\be_j$ satisfies  $|\be_j| \le s_j$. In this case a given coordinate of $\xi_j$ could be differentiated 
as many as $s_j$ times.
In this article we study multipliers that satisfy the following  coordinate-wise version of \eqref{GG1} 
\begin{equation}\begin{split}\label{GG2}
\big|\p_{\xi_{11}}^{\be_{11}}  \cdots \p_{\xi_{1n}}^{\be_{1n}}
\p_{\xi_{21}}^{\be_{21}}  \cdots  \p_{\xi_{2n}}^{\be_{2n}}
\cdots \p_{\xi_{m1}}^{\be_{m1}} &\cdots \p_{\xi_{mn}}^{\be_{mn}} \si (\xi_1,\dots , \xi_m) \big|
\\
& \lesssim (|\xi_1|+\cdots +|\xi_m|)^{-\sum_{j=1}^m
\sum_{\ell =1}^n  \be_{j\ell}},
\end{split} \end{equation}
where    $\xi_j=(\xi_{j1}, \dots , \xi_{jn})$ and each $  \be_{j\ell}$ is at most $ s_j /n$.
  Condition \eqref{GG2}
  weakens the  Coifman-Meyer hypothesis \eqref{mi} and also \eqref{GG1}  in the sense that
  it does not allow any one-dimensional variable to be differentiated more
   than an appropriate number of  times.

We now state our first main result  concerning the operator $T_\si$ in \eqref{Ts}. Here and throughout the 
$i$th coordinate of the vector  $\xi_j$ in $\R^n$ is denoted by $\xi_{ji}$. 
We denote partial derivatives in the $\xi_{ji}$ variable   by $\p_{\xi_{ji}}$. 
Also the Laplacian $\De_{\xi_j} $ on $\R^n$ is given by $ \p_{\xi_{j1}}^2+\cdots + \p_{\xi_{jn}}^2$. 
We have a result that extends condition \eqref{GG2} 
in the Sobolev space setting. 
%{\color{red}% It looks like we need a definition of 
We define
$(I-\p_{\xi_{i\ell}}^2)^{\f{\ga_{i\ell}}{2}}f(\xi)
$
as the linear operator $ ((1+4\pi^2|\eta_{i\ell}|^2)^{\f{\ga_{i\ell}}{2}}\wh f(\eta))^{\vee}(\xi)$ 
related to the multiplier
$(1+4\pi^2|\eta_{i\ell}|^2)^{\f{\ga_{i\ell}}{2}}$.
%}

\begin{theorem}\label{1dil}
Suppose that  $ 1 \le r\le 2$   and $  \ga_{i\ell}  >1/r$ for all $1\le i \le m$ and $1\le \ell \le n$. Let 
$\sigma$ be a bounded function on $\mathbb{R}^{mn}$ such that
\begin{equation}\label{dilationj}
\sup_{j\in\mathbb{Z}}\bigg\|{ \prod_{\substack{ 1\le i \le m \\ 1\le \ell \le n  }} (I-\p_{\xi_{i\ell}}^2)^{\f{\ga_{i\ell}}{2}}  \big[ \sigma(2^j\cdot)\widehat{\Psi} \big]}\bigg\|_{L^r (\mathbb{R}^{mn})}=A<\infty,
\end{equation}
where $\widehat{\Psi}$ is a smooth function  supported in the annulus $\frac{1}{2}\le |(\xi_1,\dots , \xi_m) |\le 2$
in   $\mathbb R^{mn}$  that satisfies
$$
\sum_{j\in\mathbb{Z}}\widehat{\Psi}(2^{-j}(\xi_1,\dots , \xi_m) )=1, \qquad \textup{for
all $ (\xi_1,\dots , \xi_m) \in (\mathbb R^n)^m\setminus\{0\}.$}
$$
If $1<p_i<\infty$, $i=1,\dots , m$, satisfy $\max\limits_{1\le i\le m} \max\limits_{1\le \ell\le n}   \frac{1}{\gamma_{i\ell}}  < \min\limits_{1\le i \le m} p_i $   and $\frac{1}{p}=\frac{1}{p_1}+\cdots+\frac{1}{p_m} $, then   we have
\begin{equation}\label{equ:TSigmaEST}
\norm{T_{\sigma}}_{L^{p_1}(\mathbb{R} ^{n})\times\cdots\times L^{p_m}(\mathbb{R} ^{n})\to L^p(\mathbb{R} ^{n})}\lesssim A.
\end{equation}
\end{theorem}

Taking $\ga_{i\ell}=\ga_i/n$ for all $\ell \in \{1,\dots , n\}$ and using simple embeddings between Sobolev spaces we   
deduce the following consequence of  Theorem~\ref{1dil}.

\begin{cor}\label{cor1}
Let  $ 1 \le r \le 2$   and suppose that  $\ga_{i} >n/r$ for all $i=1,\dots , m$. Let 
$\sigma$ be a bounded function on $\mathbb{R}^{mn}$ such that
\begin{equation}\label{dilationjj}
\sup_{j\in\mathbb{Z}}
\Big\|{   (I- \De_{\xi_1})^{\f{\ga_1}{2} } \cdots  (I- \De_{\xi_m})^{\f{\ga_m}{2} } 
 \big[ \sigma(2^j\cdot)\widehat{\Psi} \big]}\Big\|_{L^r (\mathbb{R}^{mn})}=A<\infty,
\end{equation}
where $\Psi$ is as in Theorem~\ref{1dil}. Then \eqref{equ:TSigmaEST}  holds where $p_i$ are as in 
Theorem~\ref{1dil}. 
\end{cor}

We also provide  an endpoint case of Corollary~\ref{cor1}  when all $p_i=1$. 
Let $H^1(\R^n)$ denote  the classical Hardy space on $\R^n$. We note that when $m=1$, boundedness for $T_\si$ is known to hold from $H^1 $ to $L^1$.

\begin{theorem}\label{End}
 Suppose that   $ 1 \le r \le 2$  and that  $\ga_{i} >n  $ for all $i=1,\dots , m$.
Let $\sigma$ be a bounded function on $\mathbb{R}^{mn}$ which satisfies \eqref{dilationjj}. 
Then we have
\begin{equation}\label{equ:TSigmaESTH1}
\norm{T_{\sigma}}_{H^{ 1}(\R^n)\times\cdots\times H^1(\R^n)\to L^{1/m,\nf}(\R^n)}\lesssim A.
\end{equation}
\end{theorem}

\bigskip

Another extension of the Coifman-Meyer multiplier theorem is in the multiparameter setting. In this case 
\eqref{mi} is relaxed to
\begin{equation}\begin{split}\label{GG3}
\big|\p_{\xi_{11}}^{\al_{11}}  &\cdots \p_{\xi_{1n}}^{\al_{1n}}
\p_{\xi_{21}}^{\al_{21}}  \cdots  \p_{\xi_{2n}}^{\al_{2n}}
\cdots \p_{\xi_{m1}}^{\al_{m1}} \cdots \p_{\xi_{mn}}^{\al_{mn}} \si (\xi_1,\dots , \xi_m) \big|
\\
& \lesssim \,\, (|\xi_{11} |
+\cdots+|\xi_{m1} |)^{-(\al_{11}+\cdots +\al_{m1})}\cdots
(|\xi_{1n} |+\cdots+|\xi_{mn} |)^{-(\al_{1n}+\cdots +\al_{mn})}
\end{split} \end{equation}
for sufficiently large  indices $\al_{i\ell} $. Such a condition was first
considered  by Muscalu, Pipher, Tao, and Thiele     \cite{MuPiTaTh, MuPiTaTh2}, who obtained boundedness for the
associated operator in the case $m=2$, i.e., from
$L^{p_1}\times   L^{p_2}$ to $L^p$ when $1/p_1+ 1/p_2=1/p$ and $1/2<p<\nf$.

In this article we also prove  a multilinear multiplier theorem   that extends   condition 
 \eqref{GG3}.  Precisely, we study multilinear multipliers that satisfy 
\begin{equation}\begin{split}\label{GG4}
\big|\p_{\xi_{11}}^{\be_{11}}  &\cdots \p_{\xi_{1n}}^{\be_{1n}}
\p_{\xi_{21}}^{\be_{21}}  \cdots  \p_{\xi_{2n}}^{\be_{2n}}
\cdots \p_{\xi_{m1}}^{\be_{m1}} \cdots \p_{\xi_{mn}}^{\be_{mn}} \si (\xi_1,\dots , \xi_m) \big|
\\
& \lesssim (|\xi_{11} |+\cdots+|\xi_{m1} |)^{-(\be_{11}+\cdots +\be_{m1})}\cdots
(|\xi_{1n} |+\cdots+|\xi_{mn} |)^{-(\be_{1n}+\cdots +\be_{mn})} 
\end{split} \end{equation}
with $\be_{ji} $ are restricted. To handle the case of  fractional  derivatives   we state our 
condition in terms of Sobolev spaces. 
We denote by 
$(I-\p_{\xi_{j\ell}}^2)^{
 \f{\ga_{j\ell}}{2}}
 $
 the operator given on the Fourier transform side 
by multiplication by 
$(1+4\pi^2 |y_{j\ell}|^2) ^{
 \f{\ga_{j\ell}}{2}}
 $,
 where $y_j$ is the dual variable of $\xi_j$. 
We now state  our   multiparameter version of Theorem~\ref{1dil},
 which extends the results in \cite{MuPiTaTh, MuPiTaTh2}
for H\"ormander type multipliers with minimal smoothness in a way that avoids time-frequency analysis.

 \begin{theorem}\label{Tensor}
Let  $ 1 \le r \le 2$  and $  \ga_{i\ell}  >1/r$ for all $1\le i \le m$ and $1\le \ell\le n$. Suppose that
$\sigma$ is a bounded function on $\mathbb R^{mn}$ such that
$$
   \sup_{k_1,\ldots,k_n\in\mathbb{Z}  } \!
  \bigg\|  \! \prod_{j=1}^m (I-\p_{\xi_{j1}}^2)^{\f{\ga_{j1}}{2}} \! \cdots  \! (I-\p_{\xi_{jn}}^2)^{\f{\ga_{jn}}{2}} \Big[ \sigma 
  %\big(2^{k_1}\xi_{11} ,\dots , 2^{k_n} \xi_{1n}, \dots , 2^{k_1}\xi_{m1} ,\dots , 2^{k_n} \xi_{mn} \big)
  \big( D_{k_1,\dots , k_n}\Xi \big)
  \!   \prod_{\ell=1}^n \!\wh{\Psi_\ell} (\xi_{1 \ell}, \dots , \xi_{m \ell})
  \Big]  \bigg\|_{L^r (\mathbb{R}^{mn})}\!\!\!\! \!\! =A<\nf,
$$
where 
$$
D_{k_1,\dots , k_n}\Xi =
 \begin{bmatrix}
    \xi_{11} & \xi_{12} &   \dots  & \xi_{1n} \\
    \xi_{21} & \xi_{22} &   \dots  & \xi_{2n} \\
         \vdots & \vdots &  \ddots & \vdots \\
   \xi_{m1} & \xi_{m2} &   \dots  & \xi_{mn}
\end{bmatrix}
\begin{bmatrix}
    2^{k_1} \\
    2^{k_2}   \\
    \vdots   \\
   2^{k_n} 
\end{bmatrix},
$$
for some $\widehat{ \Psi_\ell }$ smooth functions  on $\mathbb R^{m}$ supported
 in the annulus $\frac{1}{2}\le |\eta |\le 2$ satisfying
 \begin{equation}\label{Psic}
\sum_{k\in\mathbb{Z}}\widehat{\Psi_\ell}(2^{-k}\eta)=1, \qquad \quad\textup{for all $\eta\in \mathbb R^{m}\setminus\{0\}.$}
 \end{equation}
If $1<p_i<\infty$, $i=1,\dots , m$, satisfy $\max\limits_{1\le i\le m} \max\limits_{1\le \ell\le n}   \frac{1}{\gamma_{i\ell}}  < \min\limits_{1\le i \le m} p_i $   and $\frac{1}{p}=\frac{1}{p_1}+\cdots+\frac{1}{p_m} $, then   we have
\begin{equation*}%\label{equ:TSigmaESTQQ}
\norm{T_{\sigma}}_{L^{p_1}(\mathbb{R} ^{n})\times\cdots\times L^{p_m}(\mathbb{R} ^{n})\to L^p(\mathbb{R} ^{n})}\lesssim A.
\end{equation*}
\end{theorem}

A version of Theorem~\ref{Tensor} was proved by 
 Chen and Lu \cite{CL} when $r=m=2$  and when the 
differential operator 
$(I-\p_{\xi_{j1}}^2)^{\f{\ga_{j1}}{2}} \cdots   (I-\p_{\xi_{jn}}^2)^{\f{\ga_{jn}}{2}}$ is replaced by 
$(I-\De_{\xi_j})^{\f{\ga_j}2}$, where  $\ga_j=\ga_{j1} +\cdots + \ga_{jn} $; besides allowing $r$ to be less than $2$ and $m\ge 2$, 
Theorem~\ref{Tensor} improves that of Chen and Lu \cite{CL} in the sense that only  a restricted number of derivatives falls on  each coordinate, while in  \cite{CL} all   derivatives could fall on a single coordinate $\xi_j$  of the multiplier. 
An immediate  consequence  of Theorem \ref{Tensor} is the following:

 \begin{cor}\label{less1}
Let $ \si_{\ell}(\xi_{1 \ell }, \dots , \xi_{m \ell })$ be  bounded functions on $\mathbb{R}^{m }$
for $1\le \ell \le n$. Let
$\sigma (\xi_1, \dots , \xi_m)= \prod_{\ell =1}^n
\si_{\ell}(\xi_{1 \ell }, \dots , \xi_{m \ell })$, where  $\xi_i=(\xi_{i1},\dots , \xi_{in})\in \R^n$, $1\le i\le m$.
Suppose that for some $ \ga_{i\ell}$ 
 and $r$   as in Theorem \ref{Tensor} we have
\begin{equation}\label{dilationjQQ}
\sup_{1\le \ell\le n}\sup_{k\in\mathbb{Z}}\norm{
(I-\p_{\xi_{1\ell}}^2)^{\f{\ga_{1\ell}}{2}} \cdots   (I-\p_{\xi_{m\ell}}^2)^{\f{\ga_{m\ell}}{2}} \Big[
\sigma_\ell
 (2^k\cdot)\widehat{\Psi_\ell }   \Big]}_{L^r   (\mathbb{R}^{m})}=B<\infty
\end{equation}
where $\wh{ \Psi_\ell}$ is a smooth function supported in an annulus in $\R^m$ that satisfies  \eqref{Psic}.
Then for $\max\limits_{1\le \ell \le n} (\frac{1}{\gamma_{i\ell}}, 1)<p_i<\infty$
  for all $i=1,\dots , m$  and $\frac{1}{p}=\frac{1}{p_1}+\cdots+\frac{1}{p_m} $  we have
\begin{equation*}%\label{equ:TSigmaESTQQ22}
\norm{T_{\sigma}}_{L^{p_1}(\mathbb{R} ^{n})\times\cdots\times L^{p_m}(\mathbb{R} ^{n})\to L^p(\mathbb{R} ^{n})}\lesssim B^n.
\end{equation*}
% Additionally, if all $\vec \ga >1$, then we have that
%\begin{equation}\label{equ:TSigmaESTH1QQ}
%\norm{T_{\sigma}}_{H^{ 1}\times\cdots\times H^1\to L^{1/m,\nf}}\lesssim A.
%\end{equation}
\end{cor}

\iffalse
We also have a  version of Corollary \ref{less1} in which fewer than $m$ variables appear in each function $\si_\ell$.
\begin{cor}\label{less}
For each $1\le \ell\le n$ let    $m_\ell\le m$ and let $\si_\ell$ be bounded functions on $\R^{m_\ell}$.
If $\si(\xi_1,\dots,\xi_m)=\prod_{\ell=1}^n\si_\ell( \xi_{S_\ell , \ell})$
with $ \xi_{S_\ell ,\ell}=(\xi_{a_1\ell},\xi_{a_2\ell},\dots, \xi_{a_{m_\ell \ell }})$
and $S_\ell=\{a_1,a_2,\dots,a_{m_\ell}\} $ is a subset of $ \{1,2,\dots ,m\}$ of size $m_\ell$.
Suppose that for some $\vec \ga$, {\color{blue} $\Psi$ } and $r$   as in Corollary {\color{blue}  \ref{less1}} we have
%$\sup_j\|\si_l(2^j\cdot)\wp\|_{L^r_{\ga}}(\R^{m_l})<\nf$.
$$
A=\sup_{1\le \ell  \le n}\sup_{j\in \mathbb Z} \|\si_\ell (2^j\cdot){\color{blue} \wh \wp}\|_{L^r_{
{\color{blue} \ga_{1\ell}, \dots , \ga_{m\ell} }({\color{blue} \R^m } )}}<\nf .
$$
Then we have
$$
\norm{T_{\sigma}}_{L^{p_1}(\mathbb{R} ^{n})\times\cdots\times L^{p_m}(\mathbb{R} ^{n})\to L^p(\mathbb{R} ^{n})}\lesssim A^n.
$$
\fi

% {\color{red}
As an application, we use this corollary to give a short proof of the boundedness
of Calder\'on-Coifman-Journ\'e commutators (Proposition \ref{th7.1})
where the results in  \cite{CL, MuPiTaTh, MuPiTaTh2} are not applicable.

%}

Finally, we use arrows to denote elements of $\R^{nm}$, i.e., $\vec \xi = (\xi_1,\dots , \xi_m)$, 
 where
 $\xi_j\in \R^n$.

\medskip

\section{Preliminaries}

The following lemma will be useful in the sequel. 	
	
{\begin{lemma}\label{XL1}
Let $\ga_{i\ell},\ga_j,\ga>0$, $1\le i \le m $, $1\le j,\ell\le n$. 
Let $D^{\Gamma}$ be a differential operator  on $\mathbb R^{mn}$ 
of one of the following three  types:
\begin{eqnarray*}
 &&\prod_{i=1}^m\prod_{\ell=1}^n (I-\p_{\xi_{i\ell} }^2)^{\f{\ga_{i\ell}}{2}}; \\
  &&(I-\De_{\xi_1})^{\f{\ga_1}{2}} \cdots  (I-\De_{\xi_m})^{\f{\ga_m}{2}}; \\
  &&(I-\De_{\xi_1}-\cdots -\De_{\xi_m})^{\f{\ga }{2}}.
\end{eqnarray*}
Let $1<\rho \le r<\infty $ and let $\phi$ be a smooth function with compact support. Then there is a constant 
$C= C(\rho,r,\phi,n,\ga_{i\ell},\ga_j,\ga)$ such that 
\begin{equation}\label{Lemm11ineq}
\big\| D^{\Gamma}(\phi f) \big\|_{L^\rho( \R^{mn})} \le C \big\| D^{\Gamma}(f) \big\|_{L^r(\R^{mn})}
\end{equation}
is valid for all Schwartz functions  $f$ on $\R^{mn}$. 

Moreover, if   $D_\de$ is an operator of one of the following three  types:
 \begin{eqnarray*}
 &   & \prod_{i=1}^m\prod_{\ell=1}^n (I-\p_{\xi_{i\ell} }^2)^{\f{\de}{2}} \\
 &   & (I-\De_{\xi_1})^{\f{\de}{2}} \cdots  (I-\De_{\xi_m})^{\f{\de}{2}} \\
  &   & (I-\De_{\xi_1}-\cdots -\De_{\xi_m})^{\f{\de }{2}}
\end{eqnarray*}
then for $D^\Gamma$ and  $D_\de$ of the same type and $\de>0$ we have 
\begin{equation}\label{Lemm11ineq2}
\big\|D^{\Gamma} D_{-\de}(\phi f) \big\|_{L^1( \R^{mn})} \le C' \big\| D^{\Gamma} D_{ \de}(f) \big\|_{L^1(\R^{mn})}
\end{equation}
for all Schwartz functions  $f$ on $\R^{mn}$. Here $C'$ is a constant depending 
on $ \phi,n,\ga_{i\ell},\ga_j,\ga,\de$. 
\end{lemma}

\begin{proof}
Estimate \eqref{Lemm11ineq} could be derived by  versions 
 of the Kato-Ponce inequality adapted to the types  of operators in question, 
such versions are given   in  \cite[Section 5]{GrOh}. 
In the case where $D^{\Gamma}= (I-\De )^{\f \gamma 2}$, the proof of \eqref{Lemm11ineq} is also 
given in \cite[Lemma 7.5.7]{GrafakosMFA}. The idea in this reference also works in this setting. 
We provide a sketch: we embed $D^{\Gamma}$ in the analytic family of differential operators $D^{z\Gamma}$ (in which all $\gamma$'s are multiplied by $z$) and  reduce matters to the inequality 
$$
\big\| D^{z\Gamma}(\phi D^{-{z\Gamma}} f) \big\|_{L^\rho(\R^{mn})} \lesssim \big\| f \big\|_{L^r(\R^{mn})}\, .
$$
Let us assume that $\ga_{i\ell}$, $\ga_j $, $\ga$ 
are rational numbers;
if the case of rational numbers is proved, then by continuity we can deduce the result for all positive numbers as follows: 
on the right of the inequality we obtain a constant that is polynomial in $ \ga_{i\ell} ,  \ga_j$ or $\ga$. But each 
function $D^\Gamma (\phi f)$ and $D^\Gamma (  f)$ is continuous in  $ \ga_{i\ell} ,  \ga_j$ or $\ga$. Using this 
continuity we obtain the conclusion for all $ \ga_{i\ell} ,  \ga_j$, $\ga$ positive reals. 

 To prove \eqref{Lemm11ineq} we
  interpolate between the cases where $z=it$ and $z=2N+it$, where $N$ is a natural number 
 and common multiple of all the  denominators of  
  $\ga_{i\ell}$, $\ga_j $, $\ga$. 
 At the  endpoint cases  $z=it$ and $z=2N+it$, the 
 $D^{it\Gamma}$ and $D^{-it\Gamma}$ are   $L^\rho$ bounded  with bounds that grow at most polynomially in $t$ 
 (and in the $\ga$'s), while  $D^{2N\Gamma}$ is expanded via Leibniz's rule. 
Applying the  H\"ormander multiplier theorem and H\"older's inequality (to estimate the  $L^\rho$ norm over the support of $\phi$ by the $L^r$ norm over the entire space) we obtain the 
 claimed assertion in the cases where $z=it$ and $z=2N+it$ with bounds that grow at most polynomially in $t$. 
Interpolation for analytic families of operators yields the claimed conclusion.

We now turn our attention to \eqref{Lemm11ineq2} which is equivalent to 
\begin{equation}\label{Lemm11ineq2-2}
\big\| D^{\Gamma} D_{-\de}\big(\phi D^{- \Gamma }D_{-\de} (f)\big) \big\|_{L^1( \R^{mn})} \le C' \big\| f \big\|_{L^1(\R^{mn})}
\end{equation}
and observe that $D_{-\de}=(D_{\de})^{-1}$. 

We embed the operator $f\mapsto D^{\Gamma} D_{-\de}\big(\phi D^{- \Gamma }D_{-\de} (f)\big)$ into the analytic family of operators 
$f\mapsto D^{z\Gamma} D_{-\de}\big(\phi D^{-z \Gamma }D_{-\de} (f)\big)$ and we obtain  \eqref{Lemm11ineq2-2} as a consequence 
of interpolation between the points $z=it$ and $z=2N+it$, where $N$ is  as before and $t$ is real. 
At the endpoint $z=it$   we have that $D^{\pm it \Gamma}D_{-\de}$ is a convolution operator with  an integrable kernel and so 
$$
\| D^{it\Gamma}D_{-\de}(\phi D^{-it\Gamma}D_{-\de}(f))\|_{L^1} \lesssim \|\phi D^{-it\Gamma}D_{-\de}(f)\|_{L^1}
\lesssim \|  D^{-it\Gamma}D_{-\de}(f)\|_{L^1} \lesssim \| f \|_{L^1}
$$
with constants bounded by polynomial expressions of the $\ga$'s and $|t|$. 
When $z=2N+it$   we have 
\begin{equation}\label{Lemm11ineq2-23}
\big\| D^{it\Gamma}D_{-\de}D^{ 2N\Gamma}(\phi D^{-it}D_{-\de}(f)) \big\|_{L^1}   \lesssim     
\big\| D^{ 2N\Gamma}(\phi D^{- 2N\Gamma-it\Gamma}D_{-\de}(f)) \big\|_{L^1} 
\end{equation}
and we expand the $D^{ 2N\Gamma}$ derivative via Leibniz's rule. Then $\| D^{ 2N\Gamma}(\phi G)\|_{L^1}$ is 
bounded by a constant multiple of a sum of  terms like $\| D^{k }(G)\|_{L^1}$ where $0\le k\le 2N\Gamma$ and 
$D^{k }$ has the same type as $D^\Gamma$. 
Each operator of the form $D^{k }D^{-2N\Gamma}D^{-it\Gamma}  D_{-\de}$
is given by convolution with an integrable kernel. 
At the end we control the right hand side of \eqref{Lemm11ineq2-23} by a constant multiple of $\|f\|_{L^1}$, with 
a constant bounded by  polynomial expressions of  the $\ga$'s and $|t|$. 
This concludes the sketch of proof. 
\end{proof}
}

We will also need a reverse square function inequality associated with Littlewood-Paley operators 
acting on each variable separately.
We denote variables in $\mathbb R^{ nl}$ by $(z_1,\dots , z_n)$, where each $z_j$ lies in $\R^l$.
Fix a smooth function $\wh\Psi$ supported in an annulus in $\R^l$
satisfying $\sum_{j\in \mathbb Z}\wh\Psi  (2^{-j}z) =1$ for all $z\neq 0$.
For $j\in \mathbb Z$, define a Littlewood-Paley operator
$$
\Delta_{j }^{(k)}(f) = \big( \wh{f} (z_1,z_2,\dots , z_n)  \wh\Psi  (2^{-j}z_k)    \big)\spcheck
$$
acting on functions $f$ on $\mathbb R^{ nl}$.
 We need the following result.

\begin{lemma}\label{GLP} %Generalized Littlewood-Paley
For $f\in L^p(\mathbb R^{ nl})$ with $1<p<\nf$   we have
\begin{equation}\label{lp1}
\Big\|\Big(\sum_{j_1}\cdots\sum_{j_n}|\Delta_{j_1}^{(1)}\cdots\Delta_{j_n}^{(n)}(f)|^2\Big)^{1/2}\Big\|_{L^p(\mathbb R^{ nl})}
\lesssim\|f\|_{L^p(\mathbb R^{ nl})}.
\end{equation}
Conversely,   for $0<p<\nf$ there exists a constant $C$ such that for any  $f$   in   $ L^2(\R^{nl})$  satisfying
\begin{equation}\label{lp0}
\|(\sum_{j_1}\cdots\sum_{j_n}|\Delta_{j_1}^{(1)}\cdots
\Delta_{j_n}^{(n)}(f)|^2)^{1/2}\|_{L^p}<\nf
\end{equation}
we have
\begin{equation}\label{lp2}
\|f\|_{L^p(\mathbb R^{ nl})}\le C \Big\|\Big(\sum_{j_1}
\cdots\sum_{j_n}|\Delta_{j_1}^{(1)}\cdots\Delta_{j_n}^{(n)}(f)|^2\Big)^{1/2}\Big\|_{L^p(\mathbb R^{ nl})}.
\end{equation}
\end{lemma}

\begin{proof}[Proof of Lemma \ref{GLP}]
%We prove Lemma \ref{GLP} in this section.
The proof of \eqref{lp1} is well known and is omitted;  see for instance
\cite[Theorem 6.1.6]{Grafakos1} when $l=1$ but the same idea works for all $l$.
So we now focus on  \eqref{lp2} which we prove inductively.
The case $n=1$ is   the reverse of the Littlewood-Paley inequality when $p>1$.
 When $n=1$ and $p\le 1$, then by  \cite[Theorem 2.2.9]{GrafakosMFA}
there is a polynomial $Q$ on $\R^l$ such that
$$
\|f-Q\|_{H^p(\mathbb R^{  l})} \lesssim \Big\|\Big(\sum_{j_1} |\Delta_{j_1 }^{(1)}
(f)|^2\Big)^{1/2}\Big\|_{L^p(\mathbb R^{  l})}<\nf \, .
$$
Since $f$ lies in $ L^2(\R^{ l})$, it follows that $f-Q$ is a locally integrable
function which lies in $H^p(\R^l)$  and thus $\|f-Q \|_{L^p}\lesssim
\|f-Q\|_{H^p(\mathbb R^{  l})}<\nf$. Therefore $Q=0$ and \eqref{lp2} follows.

Assume that the assertion is valid for $n$. We will prove the
case $n+1$. Let $r_k$ be the Rademacher functions reindexed by $k\in \mathbb Z$.
 Applying   \eqref{lp2} to $g=\sum_kf_kr_k$ we obtain
\begin{align*}
&\hspace{-.5in} \int_{\mathbb R^l}\cdots\int_{\mathbb R^l}
\bigg(\sum_k|f_k(x_1,\dots,x_{n})|^2\bigg)^{{ p/2}}dx_1\cdots dx_{n}\\
\lesssim &\int_{\R^l}\cdots\int_{\R^l}\int_0^1
\Big|\sum_kf_k(x_1,\dots,x_n)r_k(t_{n+1} )   \Big|^pdt_{n+1}dx_1\cdots dx_n\\
=&C\int_0^1\int_{\R^l}\cdots\int_{\R^l}|g(x_1,\dots,x_n)|^pdx_1\cdots dx_ndt_{n+1},
\end{align*}
where   we used the property of Rademacher functions; see
for instance \cite[Appendix C]{Grafakos1}. By the induction hypothesis, the preceding expression is bounded by
a multiple of
\begin{align*}
 &\hspace{-.1in}  \int_0^1\int_{(\mathbb R^l)^n} \bigg(\sum_{j_1}\cdots\sum_{j_n}
|\Delta_{j_1}^{(1)}\cdots\Delta_{j_n}^{(n)}g(x_1,\dots,x_n)|^2 \bigg)^{p/2}dx_1\cdots dx_ndt_{n+1}\\
\lesssim&\int_0^1\!\! \int_{(\mathbb R^l)^n} \!  \int_{[0,1]^{n }}  \!
\Big|\sum_{j_1}\cdots\sum_{j_n}
\Delta_{j_1}^{(1)}\cdots\Delta_{j_n}^{(n)}g(x_1,\dots,x_n)   \prod_{i=1}^n r_{j_i}(t_i)
\Big|^p  \! dt_1\cdots dt_n d\vec x \, dt_{n+1}\\
\approx &\int_{(\mathbb R^l)^n} \! \int_{[0,1]^{n+1 }} \!
\Big| \!\! \sum_{j_1,\dots , j_n,k} \!\!
\Delta_{j_1}^{(1)}\cdots\Delta_{j_n}^{(n)}f_k(x_1,\dots,x_n)r_k(t_{n+1})\prod_{i=1}^n r_{j_i}(t_i)
\Big|^p   \! dt_1\cdots dt_{n+1}d\vec x\\
\lesssim& \int_{(\mathbb R^l)^n} \bigg(\sum_{j_1}\cdots\sum_{j_{n }}\sum_k
|\Delta_{j_1}^{(1)}\cdots\Delta_{j_{n}}^{(n)}f_k(x_1,\dots,x_{n})|^2\bigg)^{p/2}dx_1\cdots dx_{n},
\end{align*}
once again the properties  of Rademacher functions were used and $d\vec x= dx_1\cdots dx_n$.

It follows that
\begin{align*}
&\hspace{-.1in} \int_{  (\R^l)^{n+1}}  |f(x_1,\dots,x_{n+1})|^pdx_1\cdots dx_{n+1}\\
 \lesssim &\int_{ (\R^l)^{n+1}} \sup_{t>0} \big|[\vp_t*f(x_1,\dots,x_{n})](x_{n+1})\big|^pdx_{n+1}dx_1\cdots dx_{n}\\
\lesssim &  \int_{  (\R^l)^{n+1}}   \bigg(
\sum_{j_{n+1}}|\Delta_{j_{n+1}}^{(n+1)}f(x_1,\dots,x_{n+1})|^2\bigg) ^{p/2}dx_{n+1}dx_1\cdots dx_{n}\\
\approx&\int_{  (\R^l)^{n+1}}   \bigg(
\sum_{j_{n+1}}|\Delta_{j_{n+1}}^{(n+1)}f(x_1,\dots,x_{n+1})|^2\bigg) ^{p/2}dx_1\cdots dx_{n+1}\\
\lesssim&\int_{  (\R^l)^{n+1}}   \! \bigg(\sum_{j_1}\cdots\sum_{j_n}\sum_{j_{n+1}}
|\Delta_{j_1}^{(1)}\cdots\Delta_{j_n}^{(n)}\Delta_{j_{n+1}}^{(n+1)}
f(x_1,\dots,x_n,x_{n+1})|^2\bigg)^{p/2}\! dx_1\cdots dx_ndx_{n+1},
\end{align*}
where in the last step we use the inequality in the preceding alignment.
To make this argument precise, we   work with
 finitely many terms and then then pass to limit using   Fatou's lemma.
\end{proof}

  \begin{remark}\label{11071}
In both \eqref{lp1} and \eqref{lp2} we do not need 
  the full set of variables. For example, we have
  $$
  \Big\|\Big(\sum_{j_1\in \mathbb Z}\cdots\sum_{j_q\in \mathbb Z}|\Delta_{j_1}^{(1)}\cdots\Delta_{j_q}^{(q)}(f)|^2\Big)^{1/2}\Big\|_{L^p(\mathbb R^{ nl})}
\approx\|f\|_{L^p(\mathbb R^{ nl})}
  $$
  for any $1\le q\le n$ by applying 
  Lemma \ref{GLP} to $f$ as a function of $(x_1,\dots, x_q)$.
 \end{remark}

  \begin{remark}
As a consequence of   \eqref{lp2} one can derive the following inequality:
 $$
 \|f\|_{L^p(\R^{nl}) }\leq C\|f\|_{H^p({\mathbb R}^{l}\times \cdots \times {\mathbb R}^{l})}\ \ \ \ {\rm for}\ \ f\in L^2(\R^{nl})  , \ 0<p\leq 1
 $$
where $H^p(\underbrace{ {\mathbb R}^{l}\times \cdots \times {\mathbb R}^{l}}_{\textup{$n$ times}})$ denotes  the  multiparameter Hardy space; on this see \cite{HLLW}.
\end{remark}

\medskip

 \section{The proof of Theorem \ref{1dil}}\label{0926}

\begin{proof}
  For $1\le k\ne l  \le m,$ we introduce sets
  $$
  U_{k,l} = \Big\{(\xi_1,\ldots,\xi_m)\in(\mathbb R^n)^m\ :\ \max_{j\ne k,l} |\xi_j|
  \le \frac{11}{10}|\xi_k|\le \frac{11}{50m}|\xi_{l}|\Big\}
  $$
  and
  $$
  W_{k,{l}} = \Big\{(\xi_1,\ldots,\xi_m)\in(\mathbb R^n)^m\ :\ \max_{j\ne k,{l}} |\xi_j|
  \le \frac{11}{10}|\xi_k|,\frac{1}{10m}|\xi_{\ell}|\le |\xi_k|\le 2|\xi_{l}|\Big\}.
  $$
  We now construct  smooth homogeneous  of degree zero functions  $\Phi_{k,{l}}$
  and $\Psi_{k,{l}}$     supported in $U_{k,{l}}$ and $W_{k,{l}}$, respectively, and such that
\begin{equation}%\label{deu} %decomposition of unity
  \sum_{1\le k\ne l\le m}\Big(\Phi_{k,{l}}(\xi_1,\dots , \xi_m )+\Psi_{k,{l}}(\xi_1,\dots , \xi_m)\Big)=1
\end{equation}
for   every $(\xi_1,\dots , \xi_m)$
in $\big(\mathbb R^{n}\big)^m\setminus\{0\}; $ such functions can be constructed following
 the hint of Exercise 7.5.4 in  \cite{GrafakosMFA}.
In the support of $\Phi_{k,l}$  the   vector with the largest magnitude is $\xi_{l}$, while in the support
of $\Psi_{k,l}$  the    vector with the largest magnitude is $\xi_{l} $ and the one with the second largest magnitude  is
$\xi_k$. This partition of unity induces the following decomposition of $\si$:
\begin{equation}\label{12.5.bvhuawe}
 \si = \sum_{j=1}^m\sum_{\substack{k=1\\k\neq j}}^m
 \big(\si \, \Phi_{j,k} +   \si \, \Psi_{j,k} \big)\, .
 \end{equation}
 We will prove the required assertion for each piece of
 this decomposition, i.e., for    the
 multipliers $ \si \, \Phi_{j,k}$ and $ \si \, \Psi_{j,k}$ for each pair $(j,k)$ in the previous sum.
 In view of the symmetry of the decomposition, it suffices to consider the case of a  fixed  pair $(j,k)$ in the   sum in \eqref{12.5.bvhuawe}.
 To simplify {the} notation, we fix the {pair $(m,m-1)$;} thus, for the rest of the proof we fix $j=m$ and {$k=m-1$,} and we prove boundedness for the  $m$-linear operators whose symbols are
$ \si_1 =\si \, \Phi_{m,m-1}$ and $ \si_2= \si \, \Psi_{m,m-1}$.
These correspond to the $m$-linear operators
$T_{\si_1}$ and $T_{\si_2}$, respectively. Note     that $\si_1$ is supported in the set where
\[
\max(|\xi_1|, \dots , |\xi_{m-2}|) \le \tf{11}{10}\,  |\xi_{m-1}|
\qquad
\textup{and}
\qquad
 |\xi_{m-1}| \le    \tf{1 }{5m}\,  |\xi_m| \, .
\]
Also $\si_2$ is supported in the set where
\[
 \max(|\xi_1|, \dots , |\xi_{m-2}|) \le \tf{11}{10}\,  |\xi_{m-1}|
\qquad
\textup{and}
\qquad
\tf{1}{10m}\le \tf{|\xi_{m-1}|}{|\xi_m|}\le 2\, .
\]

 Fix a Schwartz function $\theta$ whose Fourier transform is supported in the annulus $\frac{1}{2}\le |\xi|\le 2$ and
$\sum_{j\in \mathbf Z} \widehat \theta (2^{-j}\xi)=1$ for $\xi \in \rn\setminus \{0\}. $
Associated with   $\theta$ we define the Littlewood--Paley operator
$\Delta_j^\theta(g)= g* \theta_{2^{-j}},$ where $\theta_t(x)=
t^{-n}\theta(t^{-1}x) $  for $t>0$.
The function $\theta$ can be extended to the function $\Theta$ defined on $\mathbb R^{nm}$ by setting $\widehat{\Theta}(\vec{\xi}\,)=\widehat{\Theta}(\xi_1,\ldots,\xi_m) = \widehat{\theta}(\xi_1+\cdots+\xi_m).$
For given Schwartz functions $f_j$ we have 
\begin{align*}
  \Delta^\theta_j \big(T_{\sigma_1}(f_1,&\ldots,f_m)\big)(x) \\
=& \int_{\mathbb R^{mn}}\widehat{\theta}(2^{-j}(\xi_1+\cdots+\xi_m))
\sigma_1(\vec{\xi}\,)\widehat{f_1}(\xi_1)\cdots\widehat{f_m}(\xi_m)
 e^{2\pi ix\cdot(\xi_1+\cdots+\xi_m)}
d\vec{\xi}\\
=& \int_{\mathbb R^{mn}}\widehat{\Theta}(2^{-j}\vec{\xi}\,)
\sigma_1(\vec{\xi}\,)\widehat{f_1}(\xi_1)\cdots\widehat{f_m}(\xi_m)
 e^{2\pi ix\cdot(\xi_1+\cdots+\xi_m)}
d\vec{\xi}.
\end{align*}
Note that for all $\vec \xi = (\xi_1,\ldots,\xi_m)$ in the support of the function $\widehat{\Theta}(2^{-j}\vec \xi\,)\sigma_1(\vec \xi\,),$ we always have $2^{j-2}\le |\xi_m|\le 2^{j+2}.$
Therefore we can take a Schwartz function $\eta$ whose Fourier transform
is supported in $\frac{1}{8}\le|\xi_m|\le 8$ and identical to $1$ on $\tf{1}{4}\leq |\xi_m|\leq 4$ and insert the factor $\widehat{\eta}(2^{-j}\xi_m)$ into the above integral without changing the outcome. More specifically
\begin{align*}
\Delta^\theta_j&\big(T_{\sigma_1}(f_1,\ldots,f_m)\big)(x)\\
 =&
\int_{\mathbb R^{mn}}\widehat{\Theta}(2^{-j}\vec{\xi}\,)
\sigma_1(\vec{\xi}\, )\widehat{f_1}(\xi_1)\cdots\widehat{f_{m-1}}(\xi_{m-1})\widehat{\eta}(2^{-j}\xi_m)
\widehat{f_m}(\xi_m)e^{2\pi ix\cdot(\xi_1+\cdots+\xi_m)}d\vec{\xi}.
\end{align*}
Now define
$
\wh{\Psi_*}(\vec \xi\,)= \sum_{|k|\le 4}\widehat{\Psi}(2^{ -k}\vec \xi\,)
$
and note that $\wh{\Psi_*}(2^{-j}\vec \xi\,)$ is equal to $1$
on the annulus $\big\{\vec \xi\in \mathbb R^{mn}  : \ 2^{j-4}\le |\vec \xi\,|\le 2^{j+4}\big\}$ which contains
the      support of $\sigma_1\widehat{\Theta}(2^{-j}\cdot) $. Then we   write
\begin{align*}
&\Delta^\theta_j\big(T_{\sigma_1}(f_1,\ldots,f_m)\big)(x)\\
 =&
 \int_{\mathbb R^{mn}}\widehat{\Psi_*}(2^{-j }\vec{\xi}\,)\widehat{\Theta}(2^{-j}\vec{\xi}\,)
\sigma_1(\vec{\xi}\, )\widehat{f_1}(\xi_1)\cdots\widehat{f_{m-1}}(\xi_{m-1})\widehat{\eta}(2^{-j}\xi_m)
\widehat{f_m}(\xi_m)e^{2\pi ix\cdot(\xi_1+\cdots+\xi_m)}d\vec{\xi}.
\end{align*}
Taking the inverse Fourier transform, we
obtain that
$
\Delta^\theta_j\big(T_{\sigma_1}(f_1,\ldots,f_m)\big)(x)
$
is equal to
\begin{equation}
\label{eq.2A01}
 \int_{(\rn)^m}2^{mnj } (\sigma_1^{j }\widehat{\Psi_*}\wh{\Theta} )\spcheck
 \big(2^{j}(x-y_1), \dots ,2^{j}(x-y_m)\big)
\prod_{i=1}^{m-1} f_i(y_i)\,
(\Delta_j^\eta  f_m)(y_m)\, d\vec y,
\end{equation}
where {$d\vec y= dy_1\cdots dy_m$,} and
$
\sigma_1^j(\xi_1,\xi_2,\dots ,\xi_m)
=\sigma_1(2^j\xi_1 ,2^j\xi_2,\dots ,2^j\xi_m).
$

Recall our assumptions that  $\max\limits_{1\le i \le m} \max\limits_{1\le \ell \le n} \f{1}{\ga_{i\ell}}<r$  and $\max\limits_{1\le i \le m} \max\limits_{1\le \ell \le n} \f{1}{\ga_{i\ell}}<\min(p_1,\dots , p_m)$.
  If  $r>1$ we pick $\rho $ such that $1< \rho < 2$    and   $\max\limits_{1\le i \le m} \max\limits_{1\le \ell \le n}\f{1}{\ga_{i\ell}} <\rho < \min(p_1,\dots , p_m,r ) $. If $r=1$, we set 
$\rho=1$. 

Define a weight for $(y_1,\dots , y_m)  \in (\R^n)^m$ by setting
$$
w_{\vec \ga}(y_1,\dots , y_m) = \prod\limits_{1\le i \le m}
 \prod\limits_{1\le \ell \le n} (1+4\pi^2 |y_{i\ell}|^2)^{\f{\ga_{i\ell}}{2}}\, . 
$$
  Let us first suppose that $\rho>1$.  We   have
{\allowdisplaybreaks
\begin{align*}
&| \De_j^\theta \big( T_{\sigma_{1}}( f_1 ,\dots , f_{m-1} , f_m )\big)(x)|\\
\leq&
\int_{(\rn)^m} \!\!\!\!\!
w_{\vec \ga} \big(2^j(x-y_1),\dots ,2^j(x-y_m)\big)  \,\,
| (\sigma_1^j\, \wh{\Psi_*} \wh{\Theta} )\spcheck(2^{j}(x-y_1),\dots ,2^{j}(x-y_m))|\\
&\qq\qq\qq\qq \times  \frac{2^{mnj} | 
 f_1 (y_1)\cdots  f_{m-1} (y_{m-1})( \Delta_j^\eta f_m)(y_m)|}{ w_{\vec \ga} \big(2^j(x-y_1),\dots ,2^j(x-y_m)\big)} \, d\vec y\\
\leq&
\bigg[\int_{(\rn)^m}\!\!  \big|
\big( w_{\vec \ga } \,
  (\sigma_1^j\, \wh{\Psi_*}\wh{\Theta})\spcheck
  \big) (2^{j}(x  - y_1),\dots,2^{j}(x  - y_m))
\big|^{\rho'}d\vec y\bigg]^{\frac{1}{\rho'}}\\
&\qq \times 2^{mnj}\left(\int_{(\rn)^m}\frac{ | f_1 (y_1)\cdots
 f_{m-1} (y_{m-1})   ( \Delta_j^\eta  f_m) (y_m)|^\rho}{w_{ \rho\vec \ga}
\big( 2^j(x-y_1), \dots , 2^j (x-y_m) \big)} \, d\vec y
\right)^{\frac{1}{\rho}}\\
\leq& C\left(\int_{(\rn)^m} \Big| w_{\vec  \ga   }(y_1,\dots , y_m)
  (\sigma_1^j\, \wh{\Psi})
\spcheck    (y_1,\dots ,y_m) \Big|^{\rho'}d\vec y\right)^{\frac{1}{\rho'}}\\
&\qq \times \left(\int_{(\rn)^m}\frac{2^{mnj}| f_1 (y_1)\cdots
 f_{m-1} (y_{m-1})  (\Delta_j^\eta f_m) (y_m)|^\rho}{\big( \prod_{\ell=1}^n (1+2^j|x_\ell-y_{1\ell} |)^{ \rho \ga_{1\ell} } \big)\cdots \big( \prod_{\ell=1}^n (1+2^j|x_\ell-y_{m\ell} |)^{ \rho \ga_{m\ell} } \big) }\, d\vec y
\right)^{\frac{1}{\rho}}\\
\leq& C\Big\| \prod_{i=1}^m\prod_{\ell=1}^n (I-\p_{\xi_{i\ell} }^2)^{\f{\ga_{i\ell}}{2}} (\sigma_1^j\, \wh{\Psi_*}\wh{\Theta}) \Big\|_{L^{\rho } }   \prod_{i=1}^{m-1} \left(\int_{\rn}\frac{ 2^{jn} |  f_{i} (y_{i})|^\rho}{   \prod_{\ell=1}^n (1+2^j|x_\ell-y_{i\ell} |)^{ \rho \ga_{i\ell} }        }\, dy_{i}\right)^{\frac{1}{\rho}} \\
&\qqq\qqq\qqq\qqq\times \left(\int_{\rn}\frac{ 2^{jn} | (\Delta_j^\eta
 f_m) (y_m)|^\rho}{     \prod_{\ell=1}^n (1+2^j|x_\ell-y_{m\ell} |)^{ \rho \ga_{m\ell} }     }\, dy_m\right)^{\frac{1}{\rho}}\\
\leq&C'
\Big\| \prod_{i=1}^m\prod_{\ell=1}^n (I-\p_{\xi_{i\ell} }^2)^{\f{\ga_{i\ell}}{2}} (\sigma_1^j\, \wh{\Psi_*}\wh{\Theta}) \Big\|_{L^{\rho } }   \bigg[\prod_{i=1}^{m-1}
  \mathcal M (|f_i|^\rho) (x)  ^{\frac{1}{\rho}} \bigg] \mathcal M (|\De_j^\eta f_m|^\rho) (x)  ^{\frac{1}{\rho}} 
\end{align*}} 
\noindent  where $\mathcal M$ is the strong maximal function given as $\mathcal M= M^{(1)}\circ \cdots \circ M^{(n)}$,
where $M^{(j)}$ is the Hardy-Littlewood maximal operator acting in the $j$th variable. 
Here we made use of the hypothesis    that $\ga_{i\ell} \rho >1$ and we used the Hausdorff-Young inequality, which is 
possible since $1\le \rho<2$. Now using \eqref{Lemm11ineq} we obtain 
\[
\Big\| \prod_{i=1}^m\prod_{\ell=1}^n (I-\p_{\xi_{i\ell} }^2)^{\f{\ga_{i\ell}}{2}} (\sigma_1^j\, \wh{\Psi_*}\wh{\Theta}) \Big\|_{L^{\rho } }  
\lesssim
\Big\| \prod_{i=1}^m\prod_{\ell=1}^n (I-\p_{\xi_{i\ell} }^2)^{\f{\ga_{i\ell}}{2}}
 (\sigma(2^j (\cdot)) \, \wh{\Psi_*} ) \Big\|_{ L^{r } }  \lesssim A \, . 
\]
 We now turn to the case where $r=1$ in which case $\rho=1$. We choose $\ga_{i\ell}'<\ga_{i\ell}$ and $\de>0$ such that 
$$
\f{1}{\ga_{i\ell} }=\f{1}{\ga_{i\ell}'+\de} <\f{1}{\ga_{i\ell}'} <\f{1}{\ga_{i\ell}'-\de}  < \frac{1}{r}=1
$$
for all $i,\ell$. The preceding argument with $\ga_{i\ell}'-\de$ in place of $\ga_{i\ell}$ yields that 
  is bounded by 
$$
| \De_j^\theta \big( T_{\sigma_{1}}( f_1 ,\dots ,  f_m )\big) | \le C'
\Big\| \prod_{i=1}^m\prod_{\ell=1}^n (I-\p_{\xi_{i\ell} }^2)^{\f{\ga_{i\ell}'-\de}{2}} (\sigma_1^j\, \wh{\Psi_*}\wh{\Theta}) \Big\|_{L^{1 } }   \bigg[\prod_{i=1}^{m-1}
  \mathcal M (|f_i| )     \bigg] \mathcal M (|\De_j^\eta f_m|  ) \, . 
  $$
  In view of   \eqref{Lemm11ineq2} we obtain 
\[
\Big\| \prod_{i=1}^m\prod_{\ell=1}^n (I-\p_{\xi_{i\ell} }^2)^{\f{\ga_{i\ell}'-\de}{2}} (\sigma_1^j\, \wh{\Psi_*}\wh{\Theta}) \Big\|_{L^{1} }  
\lesssim
\Big\| \prod_{i=1}^m\prod_{\ell=1}^n (I-\p_{\xi_{i\ell} }^2)^{\f{\ga_{i\ell}'+\de}{2}} (\sigma(2^j (\cdot)) \, \wh{\Psi } ) \Big\|_{L^1}  \lesssim A \, . 
\]

Thus,  we have   obtained the estimate
\begin{equation*}
 | \De_j^\theta  \big(T_{\sigma_{1}}( f_1 ,\dots , f_{m-1} , f_m ) \big)(x)| 
 \lesssim A
 \bigg[\prod_{i=1}^{m-1}
  \mathcal M (|f_i|^\rho) (x)  ^{\frac{1}{\rho}} \bigg] \mathcal M (|\De_j^\eta f_m|^\rho) (x)  ^{\frac{1}{\rho}} 
\end{equation*}
from which it follows that 
\begin{equation*}
\bigg(\sum_{j\in \mathbb Z} | \De_j^\theta T_{\sigma_{1}}( f_1 ,\dots , f_{m-1} , f_m ) |^2\bigg)^{\f12}
 \lesssim A
 \bigg[\prod_{i=1}^{m-1}
  \mathcal M (|f_i|^\rho)   ^{\frac{1}{\rho}} \bigg] 
  \bigg(\sum_{j\in \mathbb Z} \mathcal M (|\De_j^\eta f_m|^\rho)    ^{\frac{2}{\rho}} \bigg)^{\f12}\, .
\end{equation*}
The claimed bound follows by applying H\"older's inequality with exponents $p_1,\dots, p_m$ and using the 
boundedness of $\mathcal M$ on $L^{p_i/\rho}$, $i=1,\dots ,m$, and the Fefferman-Stein \cite{FS}  vector-valued 
Hardy-Littlewood maximal  function inequality on $L^{p_m/\rho}$.  (Note     $1<2/\rho \le 2$.)

 Next we  deal with $\sigma_2$. Using the notation introduced earlier, we write
\[
T_{\sigma_2}(f_1,\dots,f_{m-1},f_m)=\sum_{j\in \mathbb Z}
T_{\sigma_2}(f_1,\dots,f_{m-1},\De_j^\theta f_m )\, .
\vspace{-.1in}\]
We introduce another Littlewood--Paley operator {$\De_j^\zeta$,} which is
given on the Fourier transform by multiplying with a bump
$\wh{\zeta}(2^{-j}\xi)$, where $\wh{\zeta}$ is equal to one on the
annulus $\{\xi\in \rn:\,\, \tf1{2^k}\le |\xi|
\le 4\}$ with $\f1{2^k}\le\tf{1}{20m}$,
vanishes off the annulus $\frac{1}{2^{k+1}}\le |\xi|\le 8$,
and $\sum_{j}\wh\zeta(2^{-j}\xi)=k+3$.
The key observation in this case is that  
\begin{equation}
\label{eq.3A4}
%\sum_{j\in \mathbb Z} 
T_{\sigma_2}(f_1,\dots,f_{m-1},\De_j^\theta f_m )
=
%\sum_{j\in \mathbb Z}
T_{\sigma_2}\big(f_1,\dots , f_{m-2},\Delta_j^\zeta
 f_{m-1} ,\Delta_j^\theta f_m \big)\, .
\end{equation}

As in the previous case, we  have
\begin{align}
\notag
&\hspace{-.7in} T_{\sigma_2}\big(f_1,\dots , f_{m-2},\Delta_j^\zeta
 f_{m-1} ,\Delta_j^\theta f_m \big)(x)\\
 \label{eq.2A02}
=& \int_{(\mathbf R^{ n})^m}\!\!\!\!\!\!
{
\sigma_2(\vec \xi \, )
 \prod_{i=1}^{m-2} \widehat{f_i}(\xi_i) \,\,
\widehat{ \Delta_j^\zeta
 f_{m-1} }(\xi_{m-1}) \widehat{ \Delta_j^\theta
 f_m }(\xi_m)
e^{2\pi ix\cdot(\xi_1 +\cdots
+\xi_m)}
}
\, d\vec \xi.
\end{align}
The  integrand in the right-hand side of \eqref{eq.2A02}
is supported in
$
\frac{1}{2}  2^j\le|\xi_1|+\dots + |\xi_m|\le \frac{11m}{5}  2^j.
$
Thus one may insert the factor 
$$
\wh{\Psi_*}(2^{-j }\xi_1,\dots , 2^{-j }\xi_m)=
\sum_{|k|\le m+1}\wh{\Psi}(2^{-j-k}\xi_1,\dots , 2^{-j-k}\xi_m)
$$
in  the integrand.

A similar calculation as in the case for $\si_1$ yields the
estimate
$$
 |T_{\sigma_{2}}(f_1,\dots ,f_{m-2},
\Delta_j^\zeta f_{m-1} ,\Delta_j^\theta f_m ) |
\lesssim 
 A  \bigg(\prod_{i=1}^{m-2}
 \mathcal  M( |f_i|^\rho  )^{\f{1}{\rho}}   \bigg)
\mathcal   M( | \Delta_j^\zeta f_{m-1} |^\rho  )^{\f{1}{\rho}}
 \mathcal  M( | \Delta_j^\theta  f_m|^\rho   )^{\f{1}{\rho}}  \, .
$$

\noindent Summing over $j$ and taking $L^p$ norms {yields}
\begin{equation*}
\begin{split}
&\!\!\!\!\big\| T_{\sigma_{2}}(f_1,\dots ,f_{m-1},f_m)\big\|_{L^p(\rn)}\\
\leq &C\, A \,\Big\|     \bigg[ \prod_{i=1}^{m-2}
 \mathcal M (|f_i|^\rho)  ^{\frac{1}{\rho}} \bigg]
  \sum_{j\in \zzz}   \mathcal {M} (|\Delta_j^\theta
 f_{m-1} |^\rho  )^{\frac{1}{\rho}}   \mathcal {M} ( |\Delta_j^\eta
 f_{m} |^\rho ) ^{\frac{1}{\rho}}
\Big\|_{L^p }\\
\leq &C\, A \,\Big\| 
 \bigg[  \prod_{i=1}^{m-2}  \mathcal M (|f_i|^\rho)  ^{\frac{1}{\rho}} \bigg]
  \bigg( \sum_{j\in \zzz}     \mathcal {M} (|\Delta_j^\theta
 f_{m-1} |^\rho  )^{\frac{2}{\rho}} \bigg)^{\f12} \bigg(\sum_{j\in \zzz} \mathcal {M} ( |\Delta_j^\eta
 f_{m} |^\rho ) ^{\frac{2}{\rho}} \bigg)^{\f12} 
\Big\|_{L^p(\rn)}
\end{split}
\end{equation*}
Applying H\"older's inequality, the  
boundedness of $\mathcal M$ on $L^{p_i/\rho}$, $i=1,\dots ,m-1$, and the Fefferman-Stein \cite{FS} vector-valued 
Hardy-Littlewood maximal  function inequality on $L^{p_{m-1}/\rho}$ or on $L^{p_m/\rho}$  (noting that  
  $1<2/\rho\le 2$)
concludes the proof of the theorem.
 \end{proof}

 \begin{remark}\label{REMARK}
 In case I we obtained the estimate
 \begin{equation*}
 | \De_j^\theta  \big(T_{\sigma_{1}}( f_1 ,\dots , f_{m-1} , f_m ) \big) | \lesssim A
 \bigg[\prod_{i=1}^{m-1}
  \mathcal M (|f_i|^\rho)    ^{\frac{1}{\rho}} \bigg] \mathcal M (|\De_j^\eta f_m|^\rho)   ^{\frac{1}{\rho}} 
{\color{blue}
.}
\end{equation*}
 
 In case II we obtained the estimate
 $$
 |T_{\sigma_2}(f_1,\dots,f_{m-1},\De_j^\theta f_m )| \lesssim 
 A  \bigg(\prod_{i=1}^{m-2}
 \mathcal  M( |f_i|^\rho  )^{\f{1}{\rho}}   \bigg)
\mathcal   M( | \Delta_j^\zeta f_{m-1} |^\rho  )^{\f{1}{\rho}}
 \mathcal  M( | \Delta_j^\theta  f_m|^\rho   )^{\f{1}{\rho}} 
 {\color{blue}
.}
 $$
 
By symmetry for any $k_0\neq j_0$ in $\{1,\dots , m\}$ we have for $\si   \Phi_{j_0,k_0}$
  \begin{equation*}
 | \De_j^\theta  \big(T_{\si   \Phi_{j_0,k_0}}( f_1 ,\dots ,   f_m ) \big) | \lesssim A
 \bigg[\prod_{\substack { 1\le i \le   m \\  i \neq j_0 }}
  \mathcal M (|f_i|^\rho) ^{\frac{1}{\rho}} \bigg] \mathcal M (|\De_j^\eta f_{j_0}|^\rho)    ^{\frac{1}{\rho}} 
\end{equation*}
and for $\si   \Psi_{j_0,k_0}$
  \begin{equation*}
 |   T_{\si   \Psi_{j_0,k_0}}( f_1 ,\dots , \De_j^\theta f_{j_0} , \dots ,  f_m ) | \lesssim A
 \bigg[\prod_{\substack { 1\le i \le   m \\  i \neq j_0 \\ i\neq k_0 }}
  \mathcal M (|f_i|^\rho)   ^{\frac{1}{\rho}} \bigg] \mathcal M (|\De_j^\eta f_{j_0}|^\rho)   ^{\frac{1}{\rho}} 
  \mathcal   M( | \Delta_j^\zeta f_{k_0} |^\rho  )^{\f{1}{\rho}}. 
\end{equation*}
 \end{remark}

\section{The proof of Theorem~\ref{End}}

\begin{proof}
For $1\le k\ne l\le m,$ recall the sets $ U_{k,l}$ and $W_{k,l} $
and the functions  $\Phi_{k,l}$ and $\Psi_{k,l}$
  in the proof of Theorem \ref{1dil}.
 Letting $\sigma^1_{k,l}=\sigma\Phi_{k,l}$ and $\sigma^2_{k,l}=\sigma\Psi_{k,l} $, we write
  $$
  \sigma=\sum_{1\le k\ne l\le n}\big(\sigma^1_{k,l}+\sigma^2_{k,l}\big).
  $$
By the symmetry, it suffices to consider the case where $k=m-1$ and $l=m.$ We  establish  the claimed estimate
for $T_{\sigma_1}$ and $T_{\sigma_2}$ with
  $\sigma_1 = \sigma^1_{m-1,m}$ and $\sigma_2 = \sigma^2_{m-1,m}.$

  We first consider $T_{\sigma_1}(f_1, \dots , f_m)$, where
$f_j$ are fixed Schwartz functions.
We will   prove 
\begin{equation}\label{equ:HpTSigmEND}
\norm{\Big(\sum_j\Delta_j^\tht (T_{\sigma_1}(f_1,\ldots,f_m))|^2\Big)^{1/2}}_{L^{1/m,\nf}(\rn)}\lesssim A\norm{f_1}_{H^{ 1}(\rn)}\cdots\norm{f_m}_{H^{1}(\rn)}.
\end{equation}

Let $H^{1/m,\nf}$ denote the weak Hardy space of all bounded tempered distributions whose
smooth maximal function lies in   weak $L^{1/m}$. 
Given $0<p<\nf$,
for   $F$ in $  L^2(\rn)$ there is a polynomial $Q$ on $\rn$ such that
\begin{equation}\label{Q=0}
\norm{F-Q}_{L^{p,\nf}(\rn)}\le C_{p,n} \norm{F-Q}_{H^{p,\nf}(\rn)} \approx \Big\|\Big(\sum_j|\De_j(F)|^2\Big)^{1/2}\Big\|_{L^{p,\nf}(\rn)},
\end{equation}
 by a    result  of He \cite{He2014}.    
But the fact that $F$ lies in $L^2$ implies that $Q=0$. 
 Applying \eqref{Q=0} with $F=  T_{\sigma_1}(f_1,\ldots,f_m)$, for which 
 we observe that 
$
\| T_{\sigma_1}(f_1,\ldots,f_m)\|_{  L^2(\rn) }<\nf\,
$
for   Schwartz functions $f_j$, 
   we conclude from  
  \eqref{equ:HpTSigmEND} that \eqref{equ:TSigmaESTH1}  holds   for   $\sigma_1$.

To verify \eqref{equ:HpTSigmEND}, we recall  \eqref{eq.2A01} and 
set $\om_{\ga_i}( y) =   (1+4\pi^2  | y |^2  )^{\f{\ga_{i }}{2}}$ for $y\in \R^n$.
Choose $\ga_j'$ and $\delta>0$ such that 
$n< \gamma_i'-\delta <\gamma_i'<\ga_i'+\de=\ga_i$  for all $1\le i\le m$.

Now we   rewrite 
\begin{align}
&  |\Delta^\theta_j\big(T_{\sigma_1}(f_1,\ldots,f_m)\big)(x) |   \notag \\
 \leq&
\!\! \int\limits_{(\rn)^m}\!
\Big\{  \prod_{i=1}^m  \om_{\ga_i'-\de} (2^j(x-y_i)) \Big\}
| (\sigma_1^{j }\, \widehat{\Psi_*}\wh{\Theta} )\spcheck(2^{{j }}(x\! -\! y_1),\dots ,2^{{j }}(x\! -\! y_m))|   \notag  \\
&\qq\qq  \times  \frac{2^{mn{j}} |f_1(y_1)|\cdots |f_{m-1}(y_{m-1})|| (\Delta_j^\eta f_m)(y_m)|}
{  \prod_{i=1}^m  \om_{\ga_i'-\de} (2^j(x-y_i)) } \, d\vec y  \notag\\
\lesssim&\,\,\, \Big\| \Big( \prod_{i=1}^m
\omega_{  \gamma_i'-\de}  \Big)   (\sigma_1^{j }\, \widehat{\Psi_*}\wh{\Theta} )\spcheck      \Big\|_{\li}
\bigg( \prod_{i=1}^{m-1}
  M( f_i  )(x) \bigg)
  M( \Delta_j^\eta  f_m   )(x)  \, , \label{MMMJJJ}
\end{align}
as a consequence of the fact that $  \gamma_i'-\de>n $ for all $1\le i\le m$. Here $M$ is the 
uncentered Hardy-Littlewood maximal operator. In view of the Hausdorff-Young inequality, the first factor in 
  \eqref{MMMJJJ} is bounded by 
$$
\Big\|   \prod_{i=1}^m (I-\De_{\xi_i})^{\f{\ga_i'-\de}{2}}
    \big(\sigma_1^{j }\, \widehat{\Psi_*}\wh{\Theta} \big)       \Big\|_{L^1} \lesssim 
   \Big\|   \prod_{i=1}^m (I-\De_{\xi_i})^{\f{\ga_i'+\de}{2}}
    \big(\sigma (2^{j }(\cdot)) \, \widehat{\Psi }  \big)       \Big\|_{L^1}   \lesssim A
$$ 
where the penultimate inequality is a consequence of  \eqref{Lemm11ineq2} and that $\ga_i'+\de=\ga_i$.

Thus, we proved
\begin{equation*}
 |\Delta^\theta_j\big(T_{\sigma_1}(f_1,\ldots,f_m)\big)|
 \lesssim
 A \bigg( \prod_{i=1}^{m-1}
  M( f_i ) \bigg)    M( \Delta_j^\eta  f_m  ) .
\end{equation*}
 Using the preceding inequality we obtain
\begin{eqnarray*}
& &\hspace{-.7in} 
\big\|T_{\sigma_{1}}(f_1,\dots,f_{m-1},f_m)\big\|_{H^{1/m,\nf}(\rn)}\\
&\lesssim &\Big\| \Big\{\sum_j
| \Delta^\theta_j\big(T_{\sigma_1}(f_1,\ldots,f_m)\big) |^2 \Big\}^{\f12}\Big\|_{L^{1/m,\nf}(\rn)}\\
&\lesssim & A\Big\| \Big\{\sum_j   M( \Delta_j^\eta
 f_m   )^{ {2} }\Big\}^{\f12} \Big\|_{L^{1,\nf}(\rn)}
\prod_{i=1}^{m-1}  \big\|
M( f_i )   \big\|_{L^{1,\nf}(\rn)} \\
&\lesssim & A\Big\| \Big\{\sum_j    |\Delta_j^\eta
 f_m   |  ^{ {2} }\Big\}^{\f{1}{2}} \Big\|_{L^{1}(\rn)}
\prod_{i=1}^{m-1} \|f_i \|_{L^{1}(\rn)}   \\
&\lesssim &  A
\prod_{i=1}^{m } \|f_i \|_{H^{1}(\rn)}\, .
 \end{eqnarray*}
This   proves   estimate \eqref{equ:TSigmaESTH1} for $\si_1$.

 Next we  deal with $\sigma_2$. From \eqref{eq.3A4}, we have
\[
T_{\sigma_2}(f_1,\dots,f_{m-1},f_m)=\sum_{j\in \mathbb Z}
T_{\sigma_2}(f_1,\dots,f_{m-1},\De_j^\theta f_m ),
\vspace{-.1in}\]
where 
$T_{\sigma_2}\big(f_1,\dots , f_{m-2},\Delta_j^\zeta
 f_{m-1} ,\Delta_j^\theta f_m \big)$ 
is defined in \eqref{eq.2A02}.
A similar calculation as in the case for $\si_1$ yields the
estimate
$$
 |T_{\sigma_{2}}(f_1,\dots ,f_{m-2},
\Delta_j^\zeta f_{m-1} ,\Delta_j^\theta f_m ) |
\lesssim\,\,
 A  \bigg(\prod_{i=1}^{m-2}
  M( f_i  )   \bigg)
   M( \Delta_j^\zeta  f_{m-1}   )
   M( \Delta_j^\theta  f_m   )  \, .
$$

Summing over $j$, taking $L^{1/m,\nf}$ quasinorms and applying the Littlewood-Paley
characterization of $H^1$   we deduce
\begin{align*}
 & \hspace{-.1in}\big\| T_{\sigma_{2}}(f_1,\dots ,f_{m-1},f_m)\big\|_{L^{1/m,\nf}(\rn)}\\
& \lesssim  A\Big\| \prod_{i=1}^{m-2}
  M( f_i  )
  \sum_{j\in \mathbb Z}   M\big( \Delta_j^\zeta
 f_{m-1}   \big)\  M\left( \Delta_j^\theta
 f_{m}   \right)
\Big\|_{L^{1/m,\nf}(\rn) }\\
&\lesssim
{ A}  \Big\| \Big\{\prod_{i=1}^{m-2}
   M( f_i ) \Big\}
 \Big\{\sum_{j\in \mathbb Z}    M\big( \Delta_j^\zeta
 f_{m-1}   \big) ^{ 2 }
   \Big\}^{\f12}
   \Big\{\sum_{j\in \mathbb Z}    M\left( \Delta_j^\theta
 f_{m}   \right) ^{2}
   \Big\}^{\f12}
\Big\|_{L^{1/m,\nf} (\rn)} \\
&\lesssim
{ A}\Big( \prod_{i=1}^{m-2}  \big\|
   M( f_i )  \big\|_{L^{1,\nf} } \Big)
 \Big\|   \Big\{\sum_{j\in \mathbb Z}    M\big( \Delta_j^\zeta
 f_{m-1}   \big) ^{ 2 }
   \Big\}^{\!\! \f12} \Big\|_{L^{1,\nf} }  \Big\|
   \Big\{\sum_{j\in \mathbb Z}    M\big( \Delta_j^\theta
 f_{m}   \big) ^{2}
   \Big\}^{\!\!\f12}
\Big\|_{L^{1 ,\nf}  }\\
&\lesssim
{ A}\Big( \prod_{i=1}^{m-2}  \big\|
   f_i  \big\|_{L^{1 }(\rn)} \Big)
 \Big\|   \Big\{\sum_{j\in \mathbb Z}    \big| \Delta_j^\zeta
 f_{m-1}   \big| ^{ 2 }
   \Big\}^{\f12} \Big\|_{L^{1 }(\rn)}  \Big\|
   \Big\{\sum_{j\in \mathbb Z}    \left| \Delta_j^\theta
 f_{m}   \right| ^{2}
   \Big\}^{\f12}
\Big\|_{L^{1  }(\rn) }\\
&\lesssim
{ A} \prod_{i=1}^{m }  \big\|
   f_i  \big\|_{H^{1 }(\rn)}  \, .
\end{align*}

This concludes the proof of   Theorem \ref{End}.

\end{proof}

\section{The proof of Theorem~\ref{Tensor}}
  We provide the proof of Theorem~\ref{Tensor} next,
which is similar to the proof of Theorem \ref{1dil} but could be
read independently. 

%It is notationally extremely cumbersome to write a detailed proof of Theorem~\ref{Tensor}. For this reason %

Since the detailed proof of Theorem~\ref{Tensor} is notationally cumbersome,  we first present a proof in the case where $m=4$ and $n=3$, i.e., the case of $4$ variables and $3$ coordinates. This case captures all the ideas of the general case. Then we discuss the 
general case at the end. 

Consider the following  matrix of the    coordinates of all variables:  
$$
 \begin{bmatrix}
    \xi_{11} & \xi_{12} &     \xi_{13} \\
    \xi_{21} & \xi_{22}   & \xi_{23} \\
   \xi_{31} & \xi_{32}    & \xi_{33} \\
      \xi_{41} & \xi_{42}    & \xi_{43}  
\end{bmatrix}
=
\begin{bmatrix}
    \xi_{1 }  \\
    \xi_{2 } \\
   \xi_{3 }   \\
      \xi_{4 }  
\end{bmatrix}\, . 
$$

Along each column we encounter two cases: the case where the largest coordinate is larger than all the other ones 
(case I) and the other case where the largest coordinate is comparable to the second largest (case II). 
Such a splitting along all columns produces 8 cases. We only study a representative of these 8 cases, and 
in each one of those we make an 
arbitrary  assumption about the largest variable. The case below illustrates the general one. Assume that:
\begin{itemize}
\item along column 1: case I (largest  in modulus variable is $\xi_{41} $);
\item  along column 2: case II  (largest  in modulus variable is $\xi_{42} $ and second largest is $\xi_{12}$);
\item  along column 3: case I  (largest  in modulus variable is $\xi_{23} $). 
\end{itemize}
We denote the symbol associated with this case by  
$$
\tau=\si^{41,(42,12), 23}_{I,II,I}. 
$$
This symbol is obtained   by multiplying  $\si$ by a function of the form 
$$
\Phi \Big(  \f{|\xi_{11}| }{|\xi_{41}| } ,  \f{|\xi_{21}| }{|\xi_{41}| }, \f{|\xi_{31}| }{|\xi_{41}| } \Big)
\Phi \Big(  \f{|\xi_{12}| }{|\xi_{42}| } ,  \f{|\xi_{22}| }{|\xi_{42}| }, \f{|\xi_{32}| }{|\xi_{42}| } \Big)
\Psi  \Big(   \f{|\xi_{12}| }{|\xi_{42}| }    \Big) 
\Phi \Big(  \f{|\xi_{13}| }{|\xi_{23}| } ,  \f{|\xi_{33}| }{|\xi_{23}| }, \f{|\xi_{43}| }{|\xi_{23}| } \Big)
$$
where $\Phi(u_1,u_2,u_3)$ is supported in 
 $\big[0,\f{11}{200}]\times[0, \f{11}{200}]\times[0, \f{1}{20} \big]$ while 
$\Psi(u )$ is supported in $\big[\f{1}{40}, 2\big]$; see the proof of Theorem \ref{1dil} or 
\cite{GrafakosMFA} (pages 570-571 or Exercise 7.5.4).

Fix a Schwartz function $\theta$ whose Fourier transform is supported in  $[\frac{1}{2} , 2]\cup[-2,-\f12 ]$ and satisfies 
$\sum_{j\in \mathbf Z} \widehat \theta (2^{-j}v)=1$ for $v \in \mathbb R\backslash \{0\}. $
Associated with   $\theta$ we define the Littlewood--Paley operator
$\Delta_j^{(i)}(f)= f*_i \theta_{2^{-j}},$ where $\theta_t(u)=
t^{-n}\theta(t^{-1}u) $  for $t>0$ and $*_i$ denotes the convolution in the $i$th variable.
In a Littlewood-Paley operator $\De_j^{(k)}$ the upper letter inside the parenthesis indicates the coordinate on which it acts, so $1\le k\le 3$.
We write
$$
T_{\tau}(f_1,f_2,f_3,f_4)=  \sum_{j_1} \sum_{j_2} \sum_{j_3} T_{\tau} \big(  f_1,\De_{j_3}^{(3)} f_2,f_3,  \De_{j_2}^{(2)}\De_{j_1}^{(1)}f_4\big)
$$
and we have 
\begin{align*}
&\qquad  T_\tau \big(   f_1,\De_{j_3}^{(3)} f_2,f_3,  \De_{j_2}^{(2)}\De_{j_1}^{(1)}f_4\big)(x) =\\
&  \int_{\mathbb R^{12}}   %\widehat{\theta}(2^{-j_3}(\xi_{11}+\xi_{21}+\xi_{31}+\xi_{41}))
\tau(\vec \xi\, )\widehat{f_1}(\xi_1) \widehat{\theta}(2^{-j_3} \xi_{23})  \widehat{f_2}(\xi_2)\widehat{f_3}(\xi_3) \widehat{\theta}(2^{-j_2} \xi_{42}) \widehat{\theta}(2^{-j_1} \xi_{41}) \widehat{f_4}(\xi_4)
e^{2\pi ix\cdot(\xi_1+\xi_2+\xi_3+\xi_4)} d\vec{\xi} .
\end{align*}
Since $\xi_{41}$ is the largest variable among $\xi_{11},\xi _{21},\xi_{31}, \xi_{41}$, we have that 
$$
|\xi_{41}| \le |\xi_{11}|+|\xi_{21}|+|\xi_{31}|+| \xi_{41}| \le \f{232}{200} |\xi_{41}|
$$
and   since $\xi_{42}$ is the largest variable among $\xi_{12},\xi _{22},\xi_{32}, \xi_{42}$, we have that 
$$
|\xi_{42}| \le  |\xi_{12}|+|\xi_{22}|+|\xi_{32}|+| \xi_{42}| \le \f{232}{200} |\xi_{42}|.
$$
Likewise 
$$
|\xi_{23}| \le  |\xi_{13}|+|\xi_{23}|+|\xi_{33}|+| \xi_{43}| \le   \f{232}{200} |\xi_{23}|  \, .
$$
We may therefore   insert in the preceding integral the function 
 $$
\widehat{\Omega}\big(D_{-j_1,-j_2,-j_3} (\xi_1,\xi_2,\xi_3,\xi_4)\big)  =
\widehat{\Theta }(2^{-j_1}(\xi_{11}+\xi_{21}+\xi_{31}+\xi_{41}))
%\widehat{\Theta }(2^{-j_2}(\xi_{12}+\xi_{22}+\xi_{32}+\xi_{42}))
\widehat{\Theta }(2^{-j_3}(\xi_{13}+\xi_{23}+\xi_{33}+\xi_{43})), 
$$ 
where $\widehat{\Theta} (u) =  \widehat{\theta} (u/2) +\widehat{\theta} (u) +\widehat{\theta} (2u) $; notice that $\widehat{\Theta}$ 
equals $1$ on the support of $\widehat{\theta}$. 
We denote by $\widetilde {\De}_j$ the Littlewood-Paley operators  associated to $\Theta$. 
For the same reason we may also insert the function 
$$
\begin{aligned}
&\widehat{\Psi^*}\big( D_{-j_1,-j_2,-j_3}(\xi_1,\xi_2,\xi_3),\xi_4\big) 
\\
&= \widehat{\Psi_1^*}(2^{-j_1}(\xi_{11},\xi_{21},\xi_{31},\xi_{41}))\widehat{\Psi_2^*}(2^{-j_2}(\xi_{12},\xi_{22},\xi_{32},\xi_{42}))\widehat{\Psi_3^*}(2^{-j_3}(\xi_{13},\xi_{23},\xi_{33},\xi_{43}))
\end{aligned}
$$
where 
$$
\wh{\Psi_\ell^*}(u_1,u_2,u_3,u_4)= \sum_{|k|\le 1}\widehat{\Psi_\ell}(2^{ -k} (u_1,u_2,u_3,u_4) )\, , 
$$
and $\Psi_\ell$ is as in the hypotheses of the theorem. 
Let 
$$
 D_{j_1,j_2,j_3} (\xi_1,\xi_2,\xi_3,\xi_4)=
 \left( 
  \begin{bmatrix}
    \xi_{11} & \xi_{12} &     \xi_{13} &\xi_{14} \\
    \xi_{21} & \xi_{22}   & \xi_{23} &\xi_{24} \\
   \xi_{31} & \xi_{32}    & \xi_{33} &\xi_{34} \\
      \xi_{41} & \xi_{42}    & \xi_{43}&\xi_{44} 
\end{bmatrix}
\begin{bmatrix}
   2^{j_1}  \\
   2^{j_2} \\
 2^{j_3}\\
 1
\end{bmatrix} 
\right)
$$
and
$$
 \tau^{j_1,j_2,j_3} (\xi_1,\xi_2,\xi_3,\xi_4)  
 = \tau \Big( D_{j_1,j_2,j_3} (\xi_1,\xi_2,\xi_3,\xi_4)\Big).
%\prod_{\ell=1}^3 \widehat{\Psi_\ell*}( \xi_{1\ell},\xi_{2\ell},\xi_{3\ell},\xi_{4\ell} ) \, . 
$$
 
Additionally, in case II there is  the  second largest variable which is comparable to the largest one. 
Therefore we can take a Schwartz function $\eta$ whose Fourier transform
is supported in $[\frac{1}{256}, 8] \cup [-8,-\f1{256}]$ and identical to $1$ on $[\frac{1}{128}, 4] \cup [-4,-\f1{128}]$ and insert the factor $\widehat{\eta}(2^{-j_2}\xi_{12})$ into the above integral without changing the outcome. Let us denote the 
Littlewood-Paley operator associated with $\eta$  by $\overline{\De}_j$.

We may therefore rewrite
\begin{align*}
  T_{\tau} \big(  f_1,\De_{j_3}^{(3)} f_2,f_3,  \De_{j_2}^{(2)}\De_{j_1}^{(1)}f_4\big)
\,\,= & \,\,  
 T_{\tau}  \big(\overline{\De}_{j_2}^{(2)} f_1,\De_{j_3}^{(3)} f_2,f_3,  \De_{j_2}^{(2)}\De_{j_1}^{(1)}f_4\big) \\
=&\,\, \widetilde\De_{j_1}^{(1)}\widetilde\De_{j_3}^{(3)}
 T_{\tau}  \big(\overline{\De}_{j_2}^{(2)} f_1,\De_{j_3}^{(3)} f_2,f_3,  \De_{j_2}^{(2)}\De_{j_1}^{(1)}f_4\big)\, .
\end{align*}
Manipulations with the Fourier transform give that the above can be expressed as 
\begin{align*}
&  \int_{\R^{12} }  2^{4(j_1+j_2+j_3) }\Big(\tau^{j_1,j_2,j_3}  \widehat{\Psi^*} \,\,
 \widehat{\Omega}     \Big)\spcheck   \!
 \big(  D_{j_1,j_2,j_3} (x-y_1,x-y_2,x-y_3,x-y_4) \big)    \\
&\qquad  (\overline{\De}_{j_2}^{(2)} f_1 )(y_1) (\De_{j_3}^{(3)} f_2 )(y_2)  f_3(y_3)(  \De_{j_2}^{(2)}\De_{j_1}^{(1)}f_4)(y_4)
d y_1 dy_2 dy_3 dy_4\, . 
\end{align*}
 If $r=1$, set $\rho=1$. If   $r>1$  pick   $\rho$ 
  such that $1< \rho <2 $ and that  
  $$\max\limits_{1\le i \le m} \max\limits_{1\le \ell \le n}\f{1}{\ga_{i\ell}} <\rho < \min(p_1,\dots , p_m,r).$$  

Setting $\om_{\be}( y) =   (1+4\pi^2  | y |^2  )^{\f{\be}{2}}$ for $y\in \R $, we   write 
{\allowdisplaybreaks
\begin{align*}
&  \Big| T_{\tau} \big(  f_1,\De_{j_3}^{(3)} f_2,f_3,  \De_{j_2}^{(2)}\De_{j_1}^{(1)}f_4\big) (x_1,x_2,x_3) \Big|   \notag \\
 \leq&
  \int\limits_{\R^{12}}\!
2^{\f{4( j_1+j_2+j_3)}{\rho'} } \Big\{  \prod_{i=1}^4 \prod_{\ell=1}^3   \om_{\ga_{i\ell}} (2^{j_\ell}(x_\ell -y_{i\ell})) \Big\}
(\tau^{j_1,j_2,j_3}\widehat{\Psi^*}\widehat{\Omega} )\spcheck ({ D_{j_1,j_2,j_3} (x-y_1,x-y_2,x-y_3,x-y_4)})   \\
&\qquad  \f{2^{\f{ j_1+j_2+j_3}{\rho } } (\overline{\De}_{j_2}^{(2)} f_1 ) (y_1)}  
{    \prod_{\ell=1}^3   \om_{\ga_{1\ell}} (2^{j_\ell}(x_\ell -y_{1\ell}))   } 
 \quad  \f{ 2^{\f{ j_1+j_2+j_3}{\rho } }(\De_{j_3}^{(3)} f_2 )(y_2) }
{  \prod_{\ell=1}^3   \om_{\ga_{2\ell}} (2^{j_\ell}(x_\ell -y_{2\ell})) }  \\
&\qquad \f{2^{\f{ j_1+j_2+j_3}{\rho } } f_3(y_3) }
{  \prod_{\ell=1}^3   \om_{\ga_{3\ell}} (2^{j_\ell}(x_\ell -y_{3\ell})) }
\f{2^{\f{ j_1+j_2+j_3}{\rho } } (  \De_{j_2}^{(2)}\De_{j_1}^{(1)}f_4)(y_4)}
{  \prod_{\ell=1}^3   \om_{\ga_{4\ell}} (2^{j_\ell}(x_\ell -y_{4\ell})) } 
dy_1dy_2dy_3dy_4 .
\end{align*}
}
 We now apply H\"older's inequality with exponents $\rho$ and $\rho'$ to obtain the estimate  
\begin{align}\begin{split}\label{Es78}
&  \Big| T_{\tau} \big( \overline\De_{j_2}^{(2)} f_1,\De_{j_3}^{(3)} f_2,f_3,  \De_{j_2}^{(2)}\De_{j_1}^{(1)}f_4\big) (x_1,x_2,x_3) \Big|     \\
 &
\qquad \le C  A      \mathcal M(|\overline\De_{j_2}^{(2)} f_1|^\rho)^{\f{1}{\rho}}
\mathcal M(|\De_{j_3}^{(3)} f_2|^\rho)^{\f{1}{\rho}}\mathcal M(|f_3|^\rho)^{\f{1}{\rho}}
\mathcal M(|\De_{j_2}^{(2)}\De_{j_1}^{(1)}f_4|^\rho)^{\f{1}{\rho}}, 
\end{split} \end{align}
 where we used that   $\rho\ga_{i\ell}>1$ for all $i,\ell$ and also that 
\[
\Big\| \prod_{i=1}^4\prod_{\ell=1}^3 (I-\p_{\xi_{i\ell} }^2)^{\f{\ga_{i\ell}}{2}} (\tau^{j_1,j_2,j_3}\, \wh{\Psi^*}\wh{\Omega}) \Big\|_{L^{\rho } }  
\lesssim
\Big\| \prod_{i=1}^4\prod_{\ell=1}^3 (I-\p_{\xi_{i\ell} }^2)^{\f{\ga_{i\ell}}{2}} (\tau^{j_1,j_2,j_3} \, \wh{\Psi^*} ) \Big\|_{L^{\rho } }  
\]
\[
\lesssim \Big\| \prod_{i=1}^4\prod_{\ell=1}^3 (I-\p_{\xi_{i\ell} }^2)^{\f{\ga_{i\ell}}{2}} (\sigma \circ D_{j_1,j_2,j_3}) \, \wh{\Psi^*}   \Big\|_{L^{r } }  
\lesssim A
\]
which is a consequence  of Lemma~\ref{XL1} and of the fact that $\Psi^*$ is a finite sum of $\Psi_\ell$'s.

We now use \eqref{Es78} to estimate  our operator. We write 
$$
T_{\tau}(f_1,f_2,f_3,f_4)= 
\sum_{j_1} \sum_{j_2} \sum_{j_3} \widetilde\De_{j_1}^{(1)}\widetilde\De_{j_3}^{(3)}
{ T_{\tau}} \big(\overline\De_{j_2}^{(2)} f_1,\De_{j_3}^{(3)} f_2,f_3,  \De_{j_2}^{(2)}\De_{j_1}^{(1)}f_4\big)\, .
$$

Let $\mathcal M$ denote the strong maximal function. 
For each $j_1$ and $j_3$ we have the pointwise estimate 
\begin{align*}
&\big| \widetilde\De_{j_1}^{(1)}\widetilde\De_{j_3}^{(3)} \sum_{j_2} 
T_{\tau} \big(\overline\De_{j_2}^{(2)} f_1,\De_{j_3}^{(3)} f_2,f_3,  \De_{j_2}^{(2)}\De_{j_1}^{(1)}f_4\big) \big|
\\
&\qquad \le C  A   \sum_{j_2}  \mathcal M(|\overline\De_{j_2}^{(2)} f_1|^\rho)^{\f{1}{\rho}}
\mathcal M(|\De_{j_3}^{(3)} f_2|^\rho)^{\f{1}{\rho}}\mathcal M(|f_3|^\rho)^{\f{1}{\rho}}
\mathcal M(|\De_{j_2}^{(2)}\De_{j_1}^{(1)}f_4|^\rho)^{\f{1}{\rho}} \\
&\qquad \le C  A  \Big( \sum_{j_2}  
\mathcal M (|\overline\De_{j_2}^{(2)} f_1|^\rho)^{\f{2}{\rho}}\Big)^{\f12}
\mathcal M(|\De_{j_3}^{(3)} f_2|^\rho)^{\f{1}{\rho}} \mathcal M(|f_3|^\rho)^{\f{1}{\rho}}
\Big(\sum_{j_2}  \mathcal M (|\De_{j_2}^{(2)}\De_{j_1}^{(1)}f_4|^\rho)^{\f{2}{\rho}}\Big)^{\f12}
\end{align*}

We now apply Lemma~\ref{GLP} (hypothesis \eqref{lp0} is easy to check), { more precisely by Remark \ref{11071}}, to write
$$
\big\| T_{\tau}(f_1,f_2,f_3,f_4)\big\|_{L^p} \lesssim 
\Big\| \Big(\sum_{j_1} \sum_{j_3} \Big|\widetilde\De_{j_1}^{(1)}\widetilde\De_{j_3}^{(3)} 
\sum_{j_2}   { T_{\tau}}\big(\overline\De_{j_2}^{(2)} f_1,\De_{j_3}^{(3)} f_2,f_3,  \De_{j_2}^{(2)}\De_{j_1}^{(1)}f_4\big)  \Big|^2 \Big)^{\f12} \Big\|_{L^p} 
$$
and using the preceding estimate we control this expression by 
$$
A \bigg\| 
\Big( \sum_{j_2}  
\mathcal M(|\overline\De_{j_2}^{(2)} f_1|^\rho)^{\f{2}{\rho}}\Big)^{\f12}
\Big(\sum_{j_3}  \mathcal M(|\De_{j_3}^{(3)} f_2|^\rho)^{\f{2}{\rho}} \Big)^{\f12}
\mathcal M(|f_3|^\rho)^{\f{1}{\rho}}
\Big(\sum_{j_2}\sum_{j_1}\mathcal M(|\De_{j_2}^{(2)}\De_{j_1}^{(1)}f_4|^\rho)^{\f{2}{\rho}}\Big)^{\f12}
\bigg\|_{L^p}.
$$
The required conclusion follows by applying H\"older's inequality, the Fefferman-Stein inequality
\cite{FS}, { and  Lemma~\ref{GLP}} using the facts that $1\le\rho<2$ and $\rho<p_i$ for all $i$.

\bigskip

We show now how to modify { the above} proof to obtain the general case. To do so, we introduce 
some notation. We consider the set $\{1,2,\dots , n\}$ that indexes the columns of the $m\times n$ matrix
$(\xi_{kl})_{ \{1\le k \le m, 1\le l \le n\} }$. We split the set $\{1,2,\dots , n\}$ into two pieces I and II, by 
placing $l\in I$ if the $l$th column follows in the first case (where there the largest variable dominates all the 
other ones) and placing 
$l\in II$ if the $l$th column follows in the second case (where there the largest variable and the second largest are comparable). To make the notation a bit simpler, without  loss of generality we suppose that 
$I=\{1, \dots , q\}$ and $II=\{q+1,\dots , n\}$ for some $q$. Notice that one of these sets could be empty.

Recall the notation for the Littlewood-Paley operators $\De_j^{(l)}$ as in the case $m=4$, $n=3$. For the purposes 
of this theorem we introduce a slightly more refined notation using two upper indices in $\De_j^{(k,l)}$. The first 
index shows the function $f_k$ on which $\De_j^{(k,l)}$ acts and the second one the coordinate $\xi_{kl}$
of the variable $\xi_k$ on which $\De_j^{(k,l)}$ acts.

Define a map 
$$
u:  \{1,2,\dots , n\}\to  \{1,2,\dots , m\}
$$
 such that for each $l$, $u(l)$ denotes the index such that 
$\xi_{u(l)l}$ is largest among $ \xi_{kl}$. Also define a map 
$$
\bar u : \{q+1,\dots , n\}\to \{1,2,\dots , m\}
$$
such that 
$\xi_{\bar u(l)l}$ is second largest among $ \xi_{kl}$. We always have $\bar u(l) \neq u(l)$ for all $l$ in 
 $\{q+1,\dots , n\}$. We also  define 
$$
{ \De}_{j  }^{(u(r),r )}  \vec f = { \De}_{j  }^{(u(r),r )}   (f_1,\dots , f_m) =
   (f_1,\dots , { \De}_{j }^{(r)} f_{u(r)} ,\dots , f_m) 
$$
and we extend this definition to the case where ${ \De}_{j_{i_1} }^{(u(i_1),i_1 )} \cdots { \De}_{j_{i_r} }^{(u(i_r),i_r )}$
 acts on $(f_1,\dots , f_m)$.  Additionally, we use the definitions of $\widetilde \De_j$ and $\overline \De_j$ 
 as introduced in the special case $m=4$, $n=3$.

Let $\tau$ be the multilinear multiplier associated with a given fixed mapping $u$. 
We write
\begin{eqnarray*}
& & T_\tau(f_1,\dots , f_m) \\
& =  &  
 \sum_{ j_1,\dots , j_n\in \mathbb Z } T_\tau \big[{ \De}_{j_1 }^{(u(1),1  )}\cdots { \De}_{j_n }^{(u(n),n  )}  (f_1,\dots , f_m)  \big]\\
& =  & \sum_{\substack{ j_1,\dots , j_q \in \mathbb Z  }} 
  \widetilde{ \De}_{j_{1}}^{(u(1),1)} \cdots   \widetilde{ \De}_{j_{q}}^{(u(q),q)}
\sum_{\substack{ j_{q+1},\dots , j_n\in \mathbb Z}}    T_\tau  \Big[ 
\prod_{\kappa=1}^q  { \De}_{j_\kappa}^{(u(\kappa),\kappa)} 
 \prod_{\lambda=q+1}^n  { \De}_{j_{\lambda}}^{(u(\lambda),\lambda )} 
 \overline{ \De}_{j_{ \lambda}}^{ ( \overline u(\lambda),\lambda) }  
\vec f \, \Big] \, .
\end{eqnarray*}

The estimates in the case $m=4$ and $n=3$ show that 
the term in the interior sum satisfies  
\begin{align*}
& \Big|   \widetilde{ \De}_{j_{1}}^{(u(1),1)} \cdots   \widetilde{ \De}_{j_{q}}^{(u(q),q)}
\sum_{ j_{q+1},\dots , j_n\in \mathbb Z}    T_\tau  \Big[ 
\prod_{\kappa=1}^q  { \De}_{j_\kappa}^{(u(\kappa),\kappa)} 
 \prod_{\lambda=q+1}^n  { \De}_{j_{\lambda}}^{(u(\lambda),\lambda )} 
 \overline{ \De}_{j_{ \lambda}}^{ ( \overline u(\lambda),\lambda) }  
(f_1,\dots , f_m)  \Big] \Big|   \\
 & \qquad
\lesssim A \sum_{\substack{  j_{q+1},\dots , j_n\in \mathbb Z}} 
\prod_{i=1 }^m \mathcal M \bigg( \bigg| 
\prod_{\substack{ 1\le \kappa \le q \\ \kappa \in u^{-1}[i]} }  { \De}_{j_\kappa}^{(i,\kappa)} 
 \prod_{\substack{q+1 \le \lambda \le n \\ \lambda\in u^{-1}[i]  }  } { \De}_{j_{\lambda}}^{(i,\lambda )} 
  \prod_{\substack{q+1 \le \mu \le n \\ \mu\in \overline u^{-1}[i]  }}   
 \overline{ \De}_{j_{ \mu}}^{ ( i,\mu) }  
f_i  \bigg|^\rho\bigg)^{\f1 \rho} ,
\end{align*}
where $u^{-1}[i]=\{k \in \{1,\dots , n\}:\,\, u(k)=i\}$ and  
with the understanding that if any of the index sets is empty, then the corresponding Littlewood-Paley 
operators do not appear.  
Applying the Cauchy-Schwarz inequality $m-q$ times successively for the indices $j_{q+1}, j_{q+1}, \dots , j_m$
we estimate the last displayed expression by 
\begin{align}
\label{eq.5A10}
A\prod_{i=1 }^m \bigg[ 
\sum_{\substack{ j_\lambda\in \mathbb Z\\  \lambda \in u^{-1}[i] \\ q+1\le \lambda \le n }} 
\sum_{\substack{ j_\mu \in\mathbb Z\\ \mu \in \overline u^{-1}[i] \\ q+1\le \mu \le n }} 
 \mathcal M \bigg( \bigg| 
\prod_{\substack{ 1\le \kappa \le q \\ \kappa \in u^{-1}[i]} }  { \De}_{j_\kappa}^{(i,\kappa)} 
 \prod_{\substack{q+1 \le \lambda \le n \\ \lambda\in u^{-1}[i]  }  } { \De}_{j_{\lambda}}^{(i,\lambda )} 
  \prod_{\substack{q+1 \le \mu \le n \\ \mu\in \overline u^{-1}[i]  }}   
 \overline{ \De}_{j_{ \mu}}^{ ( i,\mu) }  
f_i  \bigg|^\rho\bigg)^{\!\f2 \rho} \bigg]^{\f12}. 
\end{align}

When $I\ne\emptyset$, we use Lemma~\ref{GLP} and \eqref{eq.5A10} to obtain
\begin{align*}
&   \big\|T_\tau(f_1,\dots , f_m)  \big\|_{L^p}  \\  
& =  \bigg\| 
\sum_{\substack{ j_1,\dots , j_q \in \mathbb Z  }} 
  \widetilde{ \De}_{j_{1}}^{(u(1),1)} \cdots   \widetilde{ \De}_{j_{q}}^{(u(q),q)}
\sum_{\substack{ j_{q+1},\dots , j_n\in \mathbb Z}}    T_\tau  \Big[ 
\prod_{\kappa=1}^q  { \De}_{j_\kappa}^{(u(\kappa),\kappa)} 
 \prod_{\lambda=q+1}^n  { \De}_{j_{\lambda}}^{(u(\lambda),\lambda )} 
 \overline{ \De}_{j_{ \lambda}}^{ ( \overline u(\lambda),\lambda) }  
\vec f\,  \Big]
\bigg\|_{L^p} \\
& \lesssim   \Bigg\|\bigg[ \sum_{\substack{ j_1,\dots , j_q \in \mathbb Z  }} \Big|
\widetilde{ \De}_{j_{1}}^{(u(1),1)} \cdots   \widetilde{ \De}_{j_{q}}^{(u(q),q)}
\!\!\!\!\! \sum_{\substack{ j_{q+1},\dots , j_n\in \mathbb Z}}    T_\tau  \Big[ 
\prod_{\kappa=1}^q  { \De}_{j_\kappa}^{(u(\kappa),\kappa)} 
 \prod_{\lambda=q+1}^n  { \De}_{j_{\lambda}}^{(u(\lambda),\lambda )} 
 \overline{ \De}_{j_{ \lambda}}^{ ( \overline u(\lambda),\lambda) }
\vec f\,   \Big]  \Big|^2 \bigg]^{\f12} \Bigg\|_{L^p} \\ 
& \lesssim   A\Bigg\| \bigg[\sum_{\substack{ j_1,\dots , j_q \in \mathbb Z  }} \prod_{i=1 }^m   \bigg\{  \!\!\!\!
\sum_{\substack{ j_\lambda\in \mathbb Z\\  \lambda \in u^{-1}[i]  \\ q+1\le \lambda \le n }} 
\sum_{\substack{ j_\mu \in\mathbb Z\\ \mu \in \overline u^{-1}[i] \\ q+1\le \mu \le n }} 
\!\!\!\! \mathcal M \bigg( \bigg|  \!\!\!
\prod_{\substack{ 1\le \kappa \le q \\ \kappa \in u^{-1}[i]} }  { \De}_{j_\kappa}^{(i,\kappa)} 
\!\!\! \prod_{\substack{q+1 \le \lambda \le n \\ \lambda\in u^{-1}[i]  }  } { \De}_{j_{\lambda}}^{(i,\lambda )} 
\!\!\!  \prod_{\substack{q+1 \le \mu \le n \\ \mu\in \overline u^{-1}[i]  }}   
 \overline{ \De}_{j_{ \mu}}^{ ( i,\mu) }  
f_i  \bigg|^\rho\bigg)^{\! \f2 \rho} \bigg\}
 \bigg]^{\f12} \Bigg\|_{L^p} \\
 & \lesssim   A\Bigg\| \bigg(\prod_{i=1 }^m  
 \sum_{\substack{ j_\kappa \in \mathbb Z\\   \kappa\in u^{-1}[i] \\ 1\le \kappa \le q}}  % \bigg\{ 
\sum_{\substack{ j_\lambda\in \mathbb Z\\  \lambda \in u^{-1}[i]  \\ q+1\le \lambda \le n }} 
\sum_{\substack{ j_\mu \in\mathbb Z\\ \mu \in \overline u^{-1}[i] \\ q+1\le \mu \le n }} 
\!\!\! \mathcal M \bigg( \bigg| 
\prod_{\substack{ 1\le \kappa \le q \\ \kappa \in u^{-1}[i]} }  { \De}_{j_\kappa}^{(i,\kappa)} 
\!\!\! \prod_{\substack{q+1 \le \lambda \le n \\ \lambda\in u^{-1}[i]  }  } { \De}_{j_{\lambda}}^{(i,\lambda )} 
\!\!\!  \prod_{\substack{q+1 \le \mu \le n \\ \mu\in \overline u^{-1}[i]  }}   
 \overline{ \De}_{j_{ \mu}}^{ ( i,\mu) }  
f_i  \bigg|^\rho\bigg)^{\! \f2 \rho}% \bigg\}
 \bigg)^{\f12} \Bigg\|_{L^p} .
\end{align*} 
Otherwise, when $I=\emptyset$, from \eqref{eq.5A10} we can see that $T_\tau(f_1,\dots , f_m)$ is controlled by
\begin{align}
\label{eq.5A11}
A\prod_{i=1 }^m \bigg[ 
\sum_{\substack{ j_\lambda\in \mathbb Z\\  \lambda \in u^{-1}[i] \\ 1\le \lambda \le n }} 
\sum_{\substack{ j_\mu \in\mathbb Z\\ \mu \in \overline u^{-1}[i] \\ 1\le \mu \le n }} 
 \mathcal M \bigg( \bigg| 
 \prod_{\substack{1 \le \lambda \le n \\ \lambda\in u^{-1}[i]  }  }
  { \De}_{j_{\lambda}}^{(i,\lambda )} 
  \prod_{\substack{1 \le \mu \le n \\ \mu\in \overline u^{-1}[i]  }}   
 \overline{ \De}_{j_{ \mu}}^{ ( i,\mu) }  
f_i  \bigg|^\rho\bigg)^{\!\f2 \rho} \bigg]^{\f12}. 
\end{align}
At this point we apply   H\"older's inequality and the Fefferman-Stein inequality
\cite{FS} using the facts that $1<\rho<2$ and $\rho<p_i$ for all $i$. Then we 
{ 
control
$\big\|T_\tau(f_1,\dots , f_m)  \big\|_{L^p}$
}
by a constant multiple of 
$$
A \prod_{i=1}^m 
\Bigg\| \bigg(   \sum_{\substack{j_\kappa \in \mathbb Z\\   \kappa\in u^{-1}[i] \\ 1\le \kappa \le q}}  % \bigg\{ 
\sum_{\substack{ j_\lambda\in \mathbb Z\\  \lambda \in u^{-1}[i]  \\ q+1\le \lambda \le n }} 
\sum_{\substack{ j_\mu \in\mathbb Z\\ \mu \in \overline u^{-1}[i] \\ q+1\le \mu \le n }} 
  \bigg| 
\prod_{\substack{ 1\le \kappa \le q \\ \kappa \in u^{-1}[i]} }  { \De}_{j_\kappa}^{(i,\kappa)} 
 \prod_{\substack{q+1 \le \lambda \le n \\ \lambda\in u^{-1}[i]  }  } { \De}_{j_{\lambda}}^{(i,\lambda )} 
  \prod_{\substack{q+1 \le \mu \le n \\ \mu\in \overline u^{-1}[i]  }}   
 \overline{ \De}_{j_{ \mu}}^{ ( i,\mu) }  
f_i  \bigg|^2 % \bigg\}
 \bigg)^{\f12} \Bigg\|_{L^{p_i}} 
$$
{ 
or
$$
A \prod_{i=1}^m 
\Bigg\| \bigg(   
\sum_{\substack{ j_\lambda\in \mathbb Z\\  \lambda \in u^{-1}[i]  \\ 1\le \lambda \le n }} 
\sum_{\substack{ j_\mu \in\mathbb Z\\ \mu \in \overline u^{-1}[i] \\ 1\le \mu \le n }} 
  \bigg| 
 \prod_{\substack{1 \le \lambda \le n \\ \lambda\in u^{-1}[i]  }  } { \De}_{j_{\lambda}}^{(i,\lambda )} 
  \prod_{\substack{1 \le \mu \le n \\ \mu\in \overline u^{-1}[i]  }}   
 \overline{ \De}_{j_{ \mu}}^{ ( i,\mu) }  
f_i  \bigg|^2 % \bigg\}
 \bigg)^{\f12} \Bigg\|_{L^{p_i}} 
$$
}
and by the Littlewood-Paley theorem the last expression is bounded by $A$ times the product of the $L^{p_i}$ norms
of the $f_i$. 

{ 
\begin{remark}
We see from the proof that we do not use the property
that $\xi_{kl}\in\R$, so the same argument generalizes our result
to the case when each $f_k$ is defined on $\R^{d}$
with $\xi_{kl}\in\R^d$.
This covers 
\cite[Theorem 1.10]{CL}, as we claimed in the introduction.

\end{remark}
}

\medskip

 \section{Applications: Calder\'on-Coifman-Journ\'e commutators}
 \setcounter{equation}{0}
 
 \subsection{Calder\'on commutator} 
 In 1965 Calder\'on \cite{AC} introduced the (first-order) commutator
\begin{eqnarray}\label{e6.1}
{ \mathcal C}_1 (f;a)(x)= \textup{p.v.}
\int_{\mathbb R} \frac{A(x)-A(y)}{(x-y)^2}  f(y )  dy,
\end{eqnarray}
where $a$ is the derivative of a
 Lipschitz  function $A$ and  $f$ is a test function on the real line. 
 It is known that ${ \mathcal C}_1$ is a bounded operator in $L^p(\Bbb R), 1<p<\infty$, if $A$ is a Lipchitz function
 on $\Bbb R$ and
 $$
 \|{ \mathcal C}_1 (f;a)\|_{L^p({\Bbb R})}\leq C_p \|a\|_{L^\infty({\Bbb R})}\|f\|_{L^p({\Bbb R})}, \ \ 
 1<p<\infty.
 $$
 See  Calder\'on \cite{AC, {C78}} and Coifman-Meyer \cite{CM1} for its history.  
 
\iffalse 
 In this section we apply  Theorem~\ref{1dil}  to deduce nontrivial
bounds for the commutator  of A. Calder\'on \cite{AC}.
This operator along with its higher counterparts
first appeared in the study of the
 Cauchy integral along Lipscitz curves and in fact these led to
 the first proof of the $L^2$-boundedness of the latter.
 The first order commutator is defined as 
\begin{eqnarray*}
{ \mathcal C}_1(f, a)(x)= \textup{p.v.} \int_{\mathbb R} {A(x)-A(y)\over (x-y)^2} f(y)dy,
 \ \ \ {\rm where}\ \ A'=a.
\end{eqnarray*}
\fi

Viewed as a bilinear operator acting on the pair $(f,a)$, then 
the operator ${\mathcal C}_1$ can be written as a bilinear multiplier operator
\iffalse
\begin{eqnarray}\label{e6.1}
{ \mathcal C}_1(f;a)(x)=-\textup{P.V.} \int_{\mathbb R }  { f(y )\over (x-y)^2 }
\int_{x }^{y } a(u )du \,dy \, , \qquad x\in \mathbb R .
\end{eqnarray}
 Using the Fourier transform $\wh{g}(\xi) = \int_{\mathbb R } g(x) e^{-2\pi i x\xi}dx$, 
 we may also write $ { \mathcal C}_1(f;a) $ as
  a bilinear multiplier operator concerning $f$ and $a:$\fi
\begin{eqnarray}\label{e6.2}
\qquad  { \mathcal C}_1(f;a)(x)=
-i\pi\int_{\mathbb R}\int_{\mathbb R} \wh{f}(\xi)\, \wh{a}(\eta)\, 
\left( \sgn (\eta) \Phi\big(  \xi/\eta\big)\right) \, e^{2\pi i x(\xi+\eta)}  \,
d\xi d\eta\, ,
\end{eqnarray}
where $\Phi$ is the following Lipschitz function on the real line:
\begin{eqnarray}\label{e6.2}
\Phi (s)    = \left\{
\begin{array}{lll}
-1, & s\le -1;\\ [6pt]
1+2s, &-1<s \le 0; \\ [6pt]
1, &s>0.
\end{array}
\right.
\end{eqnarray}
\iffalse
we can reduce the boundedness of ${ \mathscr C}_1$ to that of $T_\sigma$  in \eqref{Ts} with 
\begin{eqnarray}\label{e6.3}
\sigma(\xi, \eta)= \sgn (\eta) \Phi\big(  \xi/\eta\big), \ \ \ \ (\xi, \eta)\in {\mathbb R}\times {\mathbb R}.
\end{eqnarray}
\fi

The operator ${\mathcal C}_1$ is too singular to fall under the   scope of
 multilinear Calder\'on-Zygmund theory \cite{GT2}. However   it was shown to be  bounded
 from $L^{p_1}(\RR)    \times L^{p_2}(\RR)$ to $L^p(\RR)$ when
$1<p_1, p_2<\nf$ and $({1/p_1} + {1/p_2})^{-1} =p>1/2$;
see C. Calder\'on \cite{CC}.  See also  Coifman-Meyer \cite{CM1} and
Duong-Grafakos-Yan \cite{DGY}. The boundedness of   ${ \mathcal C}_1$ on $L^p$ for $p\ge 1$
 was also studied by Muscalu \cite{Mu1}   via time-frequency analysis.

In this work we  will apply Theorem~\ref{1dil}  to
obtain a direct proof of the boundedness of ${ \mathcal C}_1 $
from $L^{p_1}(\mathbb R) \times L^{p_2}(\mathbb R)$ to
  $L^p(\mathbb R)$ in the full range of $p>1/2$.  Our proof is
 based on exploiting the (limited) smoothness of the function $\Phi$, measured in terms of a  Sobolev  space  norm.
 { A partial result using a similar idea in this direction with the restriction $p>2/3$ has been obtained by
 \cite{MiTo14}.}
 
 {  For $r\ge 1$ and $\ga>0$, we recall the Sobolev space $L^r_\ga(\mathbb R^n)$, $\ga>0$  of all functions $g$ with 
 $\|(I-\De)^{\ga/2} g\|_{L^p}<\nf $.}
 {  For $\vec\ga=(\ga_1,\dots,\ga_n)$,
we denote by $L^r_{\vec\ga}(\bbr^n)$
the class of distributions $f$ such that
$$
\bigg\|{ \prod_{ 1\le \ell \le n  } (I-\p_{\ell}^2)^{\f{\ga_{\ell}}{2}}   f }\bigg\|_{L^r (\mathbb{R}^{n})}<\infty.
$$
It is easy to verify using  {  multiplier theorems}
that $L^r_{\ga}(\bbr^n)\subset L^r_{\vec\ga}(\bbr^n)$,
where $\ga=|\vec\ga|=\ga_1+\cdots+\ga_n$.
The  spaces $L^r_{\vec\ga}(\bbr^n)$ are sometimes referred to as 
Sobolev spaces with dominating mixed smoothness in the literature,
see \cite{STr} for more details and references.}
 
 To begin, 
%In this section we prove the case $n=k=1$ of Theorem \ref{TMain} and we also provide
 % some essential lemmas   needed  for estimating the local Sobolev norm of the symbols
  %  $\sigma_k^{(1)}$ for all $k\ge 1$. For simplicity,
	%we use the variables $(\xi,\eta)$ instead of $(\xi_0,\xi_1)$ in $\sigma_1^{(1)}(\xi_0,\xi_1)$.
 we  need the following characterizations of   Sobolev norms, given by Stein \cite{St2},
 \cite[Lemma 3, p. 136]{St1}.
\begin{lemma}[Stein]\label{le6.2}
(i)
Let $0< \alpha<1$ and ${2n/(n+ 2\alpha)}<p<\infty$. Then $f\in L_{\alpha}^p(\R^n)$ if and only if
$\|f\|_{L_{\alpha}^p(\R^n)} \simeq \|f\|_{L^p(\R^n)} +\|I_{\alpha}(f)\|_{L^p(\R^n)}$ where
$$
I_{\alpha}(f)(x)=\left(\int_{\R^n} {|f(x)-f(y)|^2\over |x-y|^{n+2\alpha} }dy\right)^{1/2}.
$$

(ii)
Let $1\leq \alpha<\infty$ and $1<p<\infty$. Then $f\in L_{\alpha}^p(\R^n)$ if and only if
$f\in L_{\alpha-1}^p(\R^n)$ and for $1\le j\le n$,
 ${\partial f \over \partial x_j} \in L_{\alpha-1}^p(\R^n).$
Furthermore, we have
$$\|f\|_{L_{\alpha}^p(\R^n)} \simeq \|f\|_{L^p_{\alpha-1}(\R^n)} +
\sum_{j=1}^n\left\|{\partial f \over \partial x_j}\right\|_{L^p_{\alpha-1}(\R^n)}.
$$
\end{lemma}

\iffalse
%{\color{blue}Need to cite this paper: https://projecteuclid.org/euclid.bams/1183523864}
\begin{lemma}[Stein]\label{le6.3}
For $0< \alpha<1$ and ${2n/(n+ 2\alpha)}<p<\infty$. Then $f\in L_{\alpha}^p(\R^n)$ if and only if
$\|f\|_{L_{\alpha}^p(\R^n)} \simeq \|f\|_{L^p(\R^n)} +\|I_{\alpha}(f)\|_{L^p(\R^n)}$ where
$$
I_{\alpha}(f)(x)=\left(\int_{\R^n} {|f(x)-f(y)|^2\over |x-y|^{n+2\alpha} }dy\right)^{1/2}.
$$
\end{lemma}
\fi

Throughout this section fix   a nondecreasing  smooth function $h$ on $\R$ such that
\begin{equation}\label{eq.func.h}
 h(t)=\begin{cases}
   3, & \mbox{if } t\in [4, +\infty);  \\[2pt]
   \mbox{smooth}, & \mbox{if } t\in [2, 4); \\[2pt]
   t, & \mbox{if } t\in [1/8, 2); \\[2pt]
   \mbox{smooth}, & \mbox{if } t\in[1/32, 1/8); \\[2pt]
   1/16, & \mbox{otherwise}.
 \end{cases}
\end{equation}

  \begin{lemma}\label{var}
 Let $u$ be a function supported in the rectangle
\begin{equation}\label{SETU}
\{(y_1,y_2):\,\, |y_1|\le 101/100, 1/4\le y_2\le 7/4\}
\end{equation}
 in $\mathbb R^2$ such that $\nabla u\in L^{\nf}(\R^2)$,
 and $u(x)\in L^r_{\ga}(\R^2)$ with $1<\gamma<2$, $2/\ga<r<1/(\ga-1)$.
 Define $U(y_1,y_2) = u(y_1/h(y_2),y_2).$
 Then $U\in L^r_{\ga}(\R^2)$
 and
 $$
 \|U\|_{L^r_{\ga}(\R^2)}\le C\big(\|\nabla u\|_{L^{\nf}(\R^2)}+\|u\|_{L^r_{\ga}(\R^2)}\big).
 $$
 \end{lemma}

 \begin{proof}
 Because of Lemma  \ref{le6.2}, it suffices to show
for $\al=\ga-1$  and $2/(1+\alpha) <r<{1/\alpha}$ that $U\in L^r_1(\RR^2)$, $I_{\al}(U)\in L^r(\RR^2)$
and
$I_{\al}(\p_jU)\in L^r(\RR^2)$ with $j=1,2$.
The first assertion  follows trivially by checking
the derivatives directly while the second one is verified in a way similar
 to the third one, where we adapt an argument found in
Triebel \cite[Section 4.3]{Tri} with a suitable change of variables.

Next, we     show that $I_{\al}(\p_1U)\in L^r(\RR^2)$.
We will estimate the following expression
  \begin{eqnarray*}\label{eeeee}
\|I_{\al}(\p_1U)\|_{L^r(\RR^2)}^r=\int_{\R^2} \left(\int_{\R^2}
{|\partial_1U(y)-\partial_1U(y')|^2\over |y-y'|^{2+2\alpha} }\,dy\right)^{r/2} dy'.
 \end{eqnarray*}
 Denote by $B$ a finite ball centered at $0$ containing the support of $\p_1U$. Then
 it is easy to check that, since $\p_1U\in L^{\nf}$, $r(1+\al)=r\ga>2$,
 $$
\|I_{\al}(\p_1U)\|_{L^r(\RR^2)}^r\le C\left(\|\nabla u\|_{L^{\nf}}^r+\int_{3B}
\left(\int_{3B} {|\partial_1U(y)-\partial_1U(y')|^2\over |y-y'|^{2+2\alpha} }\,dy\right)^{r/2} dy'\right),
 $$
 where $C$ is a constant depending on $B$.

Denote $x=(x_1, x_2)$, $y=(y_1, y_2)$.
One writes   $y=\vp(x)$ and $x=\psi(y)$ in the form
\begin{eqnarray*}
  \left\{
 \begin{array}{ll}
 y_1&=\vp_1(x_1, x_2)=x_1 h(x_2),\\[2pt]
  y_2&=\vp_2(x_1, x_2)=x_2
   \end{array}
 \right.
\end{eqnarray*}
and
\begin{eqnarray*}
  \left\{
 \begin{array}{ll}
 x_1&=\psi_1(y_1, y_2)={ y_1 /h(y_2)},\\[2pt]
  x_2&=\psi_2(y_1, y_2)=y_2,
     \end{array}
 \right.
\end{eqnarray*}
where $ h $ is a function defined in \eqref{eq.func.h}.
By the change of variables $y=\varphi(x)$ with $|{\rm det} \vp'(x)|<C<\infty$, direct computations give
\begin{align*}
 {\partial_1}U(y)=& {\partial \over \partial y_1}u(\psi(y))  \cdot {1\over  h(y_2)}=:
     \partial_1{  u}(\psi(y)) \cdot {1\over  h(y_2)},
\\
 {\partial_2}U(y)=&- {\partial \over  \partial y_1}
  {  u}(\psi(y)) \cdot {y_1h'(y_2)\over h(y_2)} +{\partial \over  \partial y_2} {  u}(\psi(y))
 =:- {\partial_1} {  u}(\psi(y)) \cdot {y_1h'(y_2)\over h(y_2)} +{\partial_2} {  u}(\psi(y)),
\end{align*}
and the fact that
$|\psi(y)-\psi(y')|\le \max\{\|\nabla\psi_1\|_{\nf},\|\nabla\psi_2\|_{\nf}\}|y-y'|,$
we have
\begin{eqnarray}%\label{e6.2}
& & \hspace{-.7in}	\|I_{\al}(\p_1U)\|_{L^r(\RR^2)}^r \nonumber \\
 &\leq &C\|\nabla u\|_{L^{\nf}(\R^2)}^r+C\int_{\R^2} \left(\int_{\R^2}
 {|\partial_1U(y)-\partial_1U(y')|^2\over |\psi(y)-\psi(y')|^{2+2\alpha} }
 dy\right)^{r/2} dy'\nonumber\\
 &\leq &
  C\|\nabla u\|_{L^{\nf}(\R^2)}^r+
 C\int_{\R^2} \bigg[\int_{\R^2} {\left|   {\partial_1    {  u}(\psi(y))\over  h(y_2)}
 -{\partial_1    {  u}(\psi(y'))\over  h(y'_2)} \right|^2\over |\psi(y)-\psi(y')|^{2+2\alpha} }
  dy\bigg]^{r/2}dy'\nonumber\\
  &\leq &
  C\|\nabla u\|_{L^{\nf}(\R^2)}^r+
 C\int_{\R^2} \bigg[\int_{\R^2} {\left|   {\partial_1    {  u}(x)\over  h(x_2)}
 -{\partial_1    {  u}(x')\over  h(x'_2)} \right|^2\over |x-x'|^{2+2\alpha} }
  |{\rm det} \vp'(x)| dx\bigg]^{r/2}   |{\rm det} \vp'(x')| \, dx'\nonumber\\
  &\leq &
  C\|\nabla u\|_{L^{\nf}(\R^2)}^r+
 C\int_{\R^2} \bigg[\int_{\R^2} {\left|   {\partial_1    {  u}(x)\over h(x_2)}
 -{\partial_1    {  u}(x')\over h(x'_2)} \right|^2\over |x-x'|^{2+2\alpha} }
 dx\bigg]^{r/2} \nonumber  dx'.
 \end{eqnarray}
% Note also that for $x=(x_1, x_2)\in \supp u, 1/8\leq |x_2|\leq 2$.
 Now take $\eta(x_1,x_2)\in C_0^{\infty}(\RR^2)$
 assuming value $1$ on the support of $\p_1u$ so that
 the support of $\eta$
is just a bit larger than that of $\p_1u$, and
 $h(x_2)=x_2$
 on the support of $\eta$.
% such that $\eta(|x_2|)=1 $ when $x_2\in (1/8, 2)$
Define  ${\tilde h}(x_1,x_2)=\eta(x_1,x_2)/h(x_2) $ and
then    write
\begin{eqnarray*}
  {\partial_1    {  u}(x) \over  h(x_2)}
 -{\partial_1    {  u}(x') \over  h(x'_2)} &=&  {\partial_1    {  u}(x)   {\tilde h}(x)}
 -{\partial_1    {  u}(x')  {\tilde h}(x')} \\
&=& {[\partial_1    {  u}(x)-\partial_1    {  u}(x')] {\tilde h}(x')} +
 {\partial_1    {  u}(x)} [{  {\tilde h}(x)} -{  {\tilde h}(x')}],
\end{eqnarray*}
 which %, in combination with the following
%and notice that for $(x, x')\in G\times G,$  we have that $1/8\leq |x_i|\leq 2, i=1,2$, and then
% $$
 %{|{  {\tilde h}(x)}-{   {\tilde h}(x')}|^2
% \over |x-x'|^{2+2\alpha}}  \leq  {|{  {\tilde h}(x)}-{ {\tilde h}(x')}|^2
% \over |x_2-x'_2|^{2+2\alpha}},
% $$
  yields
      \begin{eqnarray*}
	\|I_{\al}(\p_1U)\|_{L^r(\RR^2)}^r \!\!\!\!\!\!
 & &\leq
  C\|\nabla u\|_{L^{\nf}(\R^2)}^r+
 C\int_{\R^2} \left(\int_{\R^2} {\left|   {\partial_1    {  u}(x) }
 -{\partial_1    {  u}(x') } \right|^2\over |x-x'|^{2+2\alpha} }
 dx\right)^{r/2}   dx'\\
 & & \hspace{.4in}+ C\|\nabla u\|_{L^{\nf}(\R^2)}^r\int_{\RR^2}  \left(\int_{\R^2}   {|{ {\tilde h}(x)}-{  \tilde h(x')}|^2
 \over |x-x'|^{2+2\alpha}} dx_1dx_2
 \right)^{r/2}   dx'_1dx'_2 \\
  &&\leq
  C\|\nabla u\|_{L^{\nf}(\R^2)}^r+
 C\|  {\partial_1   u}\|_{L^r_{\alpha}(\R^2)}^r+C\|\nabla u\|_{L^{\nf}}^r \|\tilde h \|_{L^r_{\alpha}(\R^2)}^r\\
 &&\le  C\left(\norm{\nabla u}_{L^\infty(\R^2)} + \norm{u}_{L^r_{\gamma}(\R^2)}\right)^r.
 \end{eqnarray*}
% where in the first inequality we used the fact that $|{\partial_1      u}(x)|\leq C.$
A similar argument as the one above shows that
   \begin{eqnarray*}
 \|I_{\al}(\p_2U)\|_{L^r(\RR^2)}^r&=& \int_{\R^2} \left(\int_{\R^2}
 {|\partial_2U(y)-\partial_2U(y')|^2\over |y-y'|^{2+2\alpha} }dy\right)^{r/2} dy'\\[3pt]
&\leq&
C\|\nabla u\|_{L^{\nf}(\R^2)}^r+
 C\|  {\partial_1  u}\|_{L^r_{\alpha}(\R^2)}^r +
 C\|  {\partial_2   u}\|_{L^r_{\alpha}(\R^2)}^r\\
 &\le & C\left(\norm{\nabla u}_{L^\infty(\R^2)} + \norm{u}_{L^r_{\gamma}(\R^2)} \right)^r.
 \end{eqnarray*}
Also, by repeating the preceding argument we obtain,
\[
\norm{I_{\alpha}(U)}_{L^r(\mathbb{R}^2)}\le C\left(\norm{u}_{L^\infty(\R^2)}+\norm{u}_{L^r_{\alpha}(\mathbb{R}^2)}\right)
 \le C\norm{u}_{L^r_{\gamma}(\R^2)},
\]
where we used the Sobolev embedding theorem in the last inequality with $\gamma r>2.$
The proof of Lemma~\ref{var} is now complete.
\end{proof}

For $g,h$ on $\R $ define a the tensor $g\otimes h$  as the following function on
 $\R^2 $ by setting $(g\otimes h)(\xi,\eta) = g(\xi)h(\eta)$.

\begin{lemma}\label{ten}%tensor product
 Let $f\in L^r_{\ga}(\R)$ supported in  $[-1,1]$,
 and $\wh{\Tht}$ is a smooth function supported in an annulus centered at $0$
 with size comparable to $1$, then we have
 $$
 \big\| f\otimes\wh\Tht\big\|_{ L^r_{\ga}(\R^2)} \le C \| f\|_{ L^r_{\ga}(\R) } \, .
 $$
 \end{lemma}

\begin{proof}
We use the same idea as in the proof of Lemma \ref{var}.
%Denote $f=\theta\Phi\in L^r_{\ga}(\RR)$, and we will show that $f(\xi)\wh\Theta(\eta)\in L^r_{\ga}(\RR^2)$.
It suffices to prove that
$f\otimes\wh\Theta\in L^r_1(\R^2)$ and that $I_{\al}(\p^{\beta}(f\otimes\wh\Theta))
\in L^r(\R^2)$ with $|\beta|=1$.
It is easy to check that $\|f\otimes\wh\Theta\|_{L^r_1}\le C\|f\|_{L^r_1}$, so we only prove  that
$I_{\al}(\p_{\xi}(f\otimes\wh\Theta))
\in L^r(\R^2)$.

Note that $f\otimes\wh\Theta$ is compactly supported and we can choose
a function $\vp(\xi,\eta)\in C^{\nf}_0(\R^2)$ assuming $1$ on the support
of $f\otimes\wh\Theta$ and therefore
$f\otimes\wh\Theta=f(\xi)\vp(\xi,\eta)\wh\Theta(\eta)\vp(\xi,\eta)$.
Then $\intrr|I_{\al}(\p_{\xi}(f\otimes\wh\Theta))|^r d\xi d\eta $ is split  into the parts
$$
\intrr\left(\intrr\f{|[f'(\xi)\vp(\xi,\eta)-f'(\xi')
\vp(\xi',\eta')]\wh\Theta(\eta')\vp(\xi',\eta')|^2}
{|(\xi,\eta)-(\xi',\eta')|^{2+2\al}}d\xi'd\eta'\right)^{r/2}d\xi d\eta
$$
and
$$
\intrr\left(\intrr\f{|f'(\xi)\vp(\xi,\eta)[\wh\Theta(\eta)\vp(\xi,\eta)-
\wh\Theta(\eta')\vp(\xi',\eta')]|^2}
{|(\xi,\eta)-(\xi',\eta')|^{2+2\al}}d\xi'd\eta'\right)^{r/2}d\xi d\eta.
$$
We prove only that the first one is finite since the latter can be proved similarly.

To prove the boundedness of the first one, we split it further via the identity
$$
f'(\xi)\vp(\xi,\eta)-f'(\xi')
\vp(\xi',\eta')=
(f'(\xi)-f'(\xi'))
\vp(\xi,\eta)
+f'(\xi')(\vp(\xi,\eta)-
\vp(\xi',\eta')).
$$
The integral containing the second part is finite because $f'$ is bounded and
$\vp\in L^r_{\ga}(\R^2)$.
For the other part, a simple change of variable
$\eta'\rightarrow(\eta-\eta')/(\xi-\xi')$ shows that it is equal to
$$
C\intrr\left(\intr\f{| f'(\xi)-f'(\xi')|^2}
{|\xi-\xi'|^{1+2\al}}d\xi'\right)^{r/2} \abs{\vp(\xi,\eta)}d\xi d\eta,
$$
which, by  Lemma~\ref{le6.2},  is bounded by
$\|f\|_{L^r_{\ga}(\R)}^r$ since  $\vp\in C^{\nf}_0(\R^2)$.
%The proof of Lemma \ref{ten} is complete.
\end{proof}

 \begin{lemma}\label{lem.PhiR1}
 Let $\gamma\in (1,2)$ and $1<r<\f{1}{\ga-1}$. Then
 $\norm{\Phi\varphi}_{L^r_\gamma(\mathbb{R})}<\infty,$
 where $\varphi$ is a smooth function with compact support, and $\Phi$ is the function in \eqref{e6.2}.
 \end{lemma}

 \begin{proof}
   To obtain the claim, we need to show that
   $D^{\ga }(\varphi \Phi )= \big((1+|\xi |^2)^{\ga/2}   \wh{ \varphi \Phi } \big)\spcheck  \in L^r(\RR)$.
 Since
 \[
 \norm{ D^\ga (\varphi \Phi )}_{L^r(\mathbb{R})}\approx \norm{\varphi\Phi}_{L^r(\mathbb{R})}
 +\norm{ \big(  |\xi |  ^{\ga }   \wh{ \varphi \Phi } \big)\spcheck}_{L^r(\mathbb{R})},
 \]
 and trivially $\varphi   \Phi \in L^r(\RR)$,
 we reduce the proof to establishing
$\big\| \big(  |\xi |  ^{\ga }   \wh{ \varphi \Phi } \big)\spcheck\big\|_{L^r(\mathbb{R})}<\infty.$ 
%where $D_\xi^{\ga }(\varphi \Phi )=  \big(  |\xi |  ^{\ga }   \wh{ \varphi \Phi } \big)\spcheck$.
By the Kato-Ponce inequality for homogeneous type \cite{CW}, \cite{MuPiTaTh}, \cite{GrOh}, it suffices to show that
$ \big(  |\xi |  ^{\ga }   \wh{  \Phi } \big)\spcheck$ lies in $L^r(\RR)$.
Indeed, for   $\ga \in (1,2)$ we write
\begin{align*}
\wh{\Phi}(\xi ) |\xi|^{\ga}  &= \f{1}{\xi} \, \xi \, \wh{\Phi}(\xi ) \, |\xi|^{\ga } =
\f{1}{2\pi i } \f{1}{\xi} \,   \wh{\Phi'}(\xi ) \, |\xi|^{\ga }\\
=&-i  \f{1}{\pi \xi  } \,   \wh{\chi_{[-1,0]}}(\xi ) \, |\xi|^{\ga }
=- i  \f{1}{\pi \xi  } \,  \f{e^{2\pi i \xi}-1 }{2\pi i \xi}      \, |\xi|^{\ga }\\
=& -i    \f{1}{\pi   } \,  \f{e^{2\pi i \xi }-1 }{2\pi i } \, |\xi|^{\ga-2}
=-\f{1}{2\pi^2}  (e^{2\pi i \xi }-1  ) \, |\xi|^{\ga-2}\, .
\end{align*}
Taking inverse Fourier transforms we obtain that
$$
\big( \wh{\Phi}(\xi ) |\xi|^{\ga}\big)\spcheck (x) = c_\ga (|x+1|^{1-\ga}-|x|^{1-\ga})
$$
and this function lies in $L^r(\RR)$ when $1<r<\f{1}{\ga-1}$ and $\ga$ is very close to $2$.
 \end{proof}

 The preceding result can be lifted to $\mathbb{R}^2$ as follows.
 \begin{lemma}\label{lem.PhiR2}
 Let $\gamma\in (1,2)$ and $1<r<\f{1}{\ga-1}$, and let ${\theta}$ be a
 function supported in $\frac{1}{2}\le \abs{\xi}\le 2$ on the real line.
 Define $U(\xi,\eta) = \Phi(\frac{\xi}{\eta})\theta(\frac{\xi}{\eta})\widehat{\psi}(\xi,\eta),$
 where $\widehat{\psi}$ is a smooth function supported in an annulus centered at zero. Then
 $\norm{U}_{L^r_\gamma(\mathbb{R}^2)}<\infty$.
 \end{lemma}
 \begin{proof}
     Set
 $$u(\xi,\eta) = \Phi(\xi ) \theta(\xi ) \wh{\Psi} (\xi\eta,\eta)$$ and
$$U(\xi,\eta)= \Phi(\xi/\eta) \theta(\xi/\eta) \wh{\Psi} (\xi,\eta).$$
Since $h(\eta)=\eta$ on the support of the function $U.$ We now apply Lemma \ref{var} to obtain
$$
\|U\|_{L^r_{\ga}(\R^2)}\le C\big(\|\nabla u\|_{L^{\nf}(\R^2)}+\|u\|_{L^r_{\ga}(\R^2)}\big).
$$
Thus, it is enough to show that $\|u\|_{L^r_{\ga}(\R^2)}<\infty.$
We introduce a compactly supported
smooth function $\wh{\Theta}(\eta)$ which is equal to $1$ on the support of
$\eta\mapsto \theta(\xi )\wh{\Psi} (\xi\eta,\eta)$ for any $\xi$.
the Kato-Ponce inequality (\cite{KP} \cite{GrOh}) allows us to
estimate the Sobolev norm of $u$ as follows:
\begin{align*}
\norm{u}_{L^r_\ga(\RR^2)}
\,\,=&\,\,\big\| \Phi(\xi )\theta(\xi ) \wh{\Theta}(\eta)\wh{\Psi} (\xi\eta,\eta)\big\|_{L^r_\ga(\RR^2)}\\
 \lesssim&\,\,
\big\|  \Phi (\xi)\theta(\xi ) \wh{\Theta}(\eta)\big\|_{L^r_\ga(\RR^2)}
 \big\|  \wh{\Psi} (\xi\eta,\eta)  \big\|_{L^\infty(\RR^2)} \\
&+
\big\|   \wh{\Psi} (\xi\eta,\eta)\big\|_{L^r_\ga (\RR^2)}
\big\|  \Phi(\xi)\theta(\xi )\wh{\Theta}(\eta) \big\|_{L^\infty(\RR^2)}\, .
\end{align*}
We are left with establishing $ \|  \Phi (\xi)\theta(\xi ) \wh{\Theta}(\eta) \|_{L^r_\ga(\RR^2)}<\nf$,
since all other terms on the right of the above inequality are finite.
This is achieved via Lemmas \ref{lem.PhiR1} and \ref{ten}. Thus  the proof of Lemma \ref{lem.PhiR2} is
complete.
 \end{proof}

 Using these ideas we are able to deduce the following result concerning  ${ \mathcal C}_1 $.

\begin{prop}\label{C1}
The Calder\'on commutator ${ \mathcal C}_1 $ maps
$L^{p_1}(\RR) \times L^{p_2}(\RR)$ to
$L^p(\RR)$ when $1/p_1+1/p_2=1/p$, $1<p_1,p_2 <\infty$, and $1/2<p<\infty$.
\end{prop}

\begin{proof}
Note that $\si(\xi,\eta)=\sgn(\eta)\Phi(\xi/\eta)$ has an obvious modification which is continuous
on $\R^2\backslash\{0\}$. We denote the latter by $\sgn(\eta)\Phi(\xi/\eta)$ as
well since there is no chance to introduce any confusion.

We introduce a smooth function with compact support $\theta$ on the real line which is supported
in two small intervals, say, of length $1/100$ centered at the points $-1$ and $0$. Then we write
$$
1= \theta (\xi/\eta) + 1-\theta (\xi/\eta)
$$
and we decompose the function $\sgn(\eta)\Phi(\xi/\eta) =\si_1(\xi,\eta)+\si_2(\xi,\eta)$, where
$\si_1(\xi,\eta)=\sgn(\eta)\Phi(\xi/\eta)\theta (\xi/\eta)$
and $\si_2(\xi,\eta)=\sgn(\eta)\Phi(\xi/\eta)(1-\theta (\xi/\eta))$.
 Let $\wh{\Psi}$ be a smooth bump supported in the annulus
$1/2<|(\xi,\eta)| <3/2$ in $\RR^2$.
The function $\si_2$ is smooth away from zero and $\si_2\wh\Psi$ lies in $L^r_{\ga}(\R^2)$ for any $r,\ga>1$
Also, $\si_1\wh\Psi$ lies in $L^r_{\ga}(\R^2)$ with
$r\ga>1$.
in view of Lemma \ref{lem.PhiR2}.  Then Corollary \ref{cor1} implies the
required conclusion.
\end{proof}

\bigskip

 \subsection{Commutators  of Calder\'on-Coifman-Journ\'e}
 Now
   we focus on the boundedness
properties   of the
following $n$-dimensional version of  ${\mathcal  C}_1 $:
\begin{align}\begin{split}\label{PTCJ}
\hspace{-29pt} &  { \mathcal C}_1^{(n)}(f,a)(x)\\
&     =\textup{p.v.}\!\!
\int_{\mathbb R^n} \! f(y) \!  \left(
\prod_{l=1}^n  { 1\over
 (y_l-x_l)^2 } \right ) \! \int_{x_1}^{y_1} \!\!\!
 \cdots \!\!\!\int_{x_n}^{y_n}  a(u_1, \dots, u_n) \, du_1\cdots  du_n \, dy,
\end{split}\end{align}
where  $f$ is a function on $\mathbb R^n$, and
 $x=(x_1,\dots, x_n)\in \mathbb R^n$, $y=(y_1, \dots , y_n) \in \mathbb R^n$.
 The operator ${ \mathcal C}_1^{(n)}$ was introduced by a suggestion of  Coifman
 when $n=2$.
 The $L^2\times L^\nf\to L^2$ bound for ${ \mathcal C}_1^{(2)}$ was studied  by Aguirre \cite{Aguirre} and
   Journ\'e \cite{Jo1, Jo2}, namely,
%In his   celebrated work    on product-type spaces, Journ\'e \cite{Jo1}   named  ${ \mathscr C}_1^{(2)}$
%{\it a bicommutator of Calder\'on-Coifman type}, and in
 % \cite{Jo2} he  showed that     
\begin{equation}\label{Journe1Rev}
\|{ \mathcal C}_1^{(2)}(f,a)\|_{L^2({\mathbb R^2})}
\leq C \|a\|_{L^{\infty}(\mathbb R^2)}\|f\|_{L^{2}(\mathbb R^2)}.
\end{equation}
For general $n\geq 2$, boundedness for $\mathcal C_1^{(n)}$ from $L^{p_1}\times L^{p_2}$ to $L^p$ for $p>1/2$,
 can be derived by
Muscalu's work  on Calder\'on
commutators on polydiscs \cite[Theorem 6.1]{Mu3}  via time-frequency analysis.

In this section we  will apply Corollary~\ref{less1} to
obtain a direct proof of the boundedness of ${ \mathcal C}_1^{(n)} $
from $L^{p_1}(\mathbb R^n) \times L^{p_2}(\mathbb R^n)$ to
  $L^p(\mathbb R^n)$ in the full range of $p>1/2$. 
  % Again, our proof is
 %based on exploiting the (limited) smoothness of the function $\Phi$, 
 %measured in terms of a  Sobolev  space  norm.

%A natural question is to    investigate  boundedness of the bicommutator  ${ \mathscr C}^{(2)}(f,a)$
%from $L^p({\mathbb R}^2)\times L^q({\mathbb R}^2)$ into  $L^r({\mathbb R}^2)$ for some
% $1\leq p\leq\infty$, $1\leq q\leq \infty$ and $0<r\leq \infty$ satisfying $1/r={1/p}+{1/q}.$
% In this section, we will show the following result.

\begin{prop}\label{th7.1}

 Let  $1<p_1, p_2<\infty$,   $1/2<p<\infty$ and $1/p={1/p_1}+{1/p_2}.$
Then the operator  ${ \mathcal C}_1^{(n)}(f,a)$ is bounded from $L^{p_1}({\mathbb R}^n)\times L^{p_2}({\mathbb R}^n)$
into $L^p({\mathbb R}^n)$, i.e.,
\begin{eqnarray*}
\|{ \mathcal C}_1^{(n)}(f,a)\|_{{L^p(\RR^n)}}\leq C_p\|a\|_{L^{p_1}(\RR^n)}\|f\|_{L^{p_2}(\RR^n)}.
\end{eqnarray*}
\end{prop}

\begin{proof} 
 The operator  ${ \mathcal C}^{(n)}_1(f,a)$
 is a bilinear operator  which can also be expressed in bilinear Fourier multiplier form as
 \begin{eqnarray*}
   { \mathcal C}^{(n)}_1(f,a)(x)
= (-i\pi)^n  \iint_{\mathbb R^n\times \mathbb R^n} \wh{f}(\xi_1,\cdots,  \xi_n)\, \wh{a}(\eta_1,\cdots, \eta_n)\,
e^{2\pi i x\cdot(\xi+\eta)  } m(\xi; \eta)
    \,
d\xi d\eta\, ,
\end{eqnarray*}
where the  symbol $m$ is given by
$$
m(\xi; \eta)=\prod_{i=1}^n\left[\sgn (\eta_i) \,
  \Phi\Big({\xi_i\over \eta_i}\Big)\right],
$$
and $\xi=(\xi_1, \cdots, \xi_n)$ and $\eta=(\eta_1, \cdots, \eta_n).$
Since $m(\xi,\eta)=\prod_{i=1}^n\si(\xi_i,\eta_i)$
is a product of $n$ equal pieces, by Corollary~\ref{less1},
it suffices to verify that
$\sup_{k\in \mathbb Z}\|\si(2^k\cdot)\wh\Psi\|_{L^r_{\ga/2,\ga/2}(\mathbb R^2)}
=B<\nf$.
Note that $\si(2^k\cdot)\wh\Psi\in L^r_{\ga}(\bbr^2)$
uniformly in $k$
by Proposition~\ref{C1}, so they are  also in 
$L^r_{\ga/2,\ga/2}(\bbr^2)$ uniformly
due to that $L^r_\ga(\bbr^2)\subset L^r_{\ga/2,\ga/2}(\bbr^2)$.
We complete the proof of Proposition \ref{th7.1}.
\end{proof}


\bigskip

\begin{thebibliography}{99}

 \bibitem{Aguirre} J. Aguirre,
 {\it Multilinear Pseudo-differential operators and paraproducts},  Thesis (Ph.D.)--Washington University in St. Louis, 1981, 155 pp.


  \bibitem{AC} A. P. Calder\'on,  {\it  Commutators of singular
 integrals},    Proc. Nat. Acad. Sci. U.S.A.   {\bf 53} (1965), 1092--1099.


 \bibitem{C78} A. P. Calder\'on,   {\it Commutators,  singular integrals on
 Lipschitz curves and applications},  Proc. Intern. Congress Math.   Helsinki    (1978), 85--96.

 \bibitem{CC} C. P. Calder\'on,   {\it  On commutators of singular
 integrals}, Studia Math.   {\bf 53} (1975), 139--174.

\bibitem{CT}  A. P.  Calder\'on and A. Torchinsky,
  {\it Parabolic maximal functions associated with a distribution,  II,}
  Adv. Math.  \textbf{24} (1977),  101--171.

  \bibitem{CL} J. Chen and G. Lu,
  {\it H\"ormander type theorems for multi-linear and
multi-parameter Fourier multiplier operators with
limited smoothness}, Nonlinear Anal. {\bf 101} (2014) 98--112.

% \bibitem{CJ} M. Christ and J.-L. Journ\'e,  {\it Polynomial growth estimates for multilinear singular integral
% operators}, Acta Math.    {\bf 159} (1987), 51--80.

  \bibitem{CW} M. Christ and M. Weinstein,
  {\it Dispersion of small-amplitude solutions of the generalized Korteweg-de Vries equation},
 J. Funct. Anal. (1991)  {\bf 100}, 87--109.

% \bibitem{CMcM} R. R. Coifman, A. McIntosh  and Y. Meyer,
%\emph{L' int\'egrale de Cauchy d\'efinit un op\'erateur born\'e sur $L^2$ pour les courbes lipschitziennes},
% Ann. of Math. (2) {\bf 116} (1982),   361--387.

 \bibitem{CM1} R. R. Coifman and Y. Meyer,    {\it On commutators of singular
 integral and bilinear singular integrals}, Trans. Amer. Math. Soc.  {\bf 212} (1975), 315--331.

\bibitem{CM-G} R. R.  Coifman and Y. Meyer,   {\it Commutateurs d' int\'egrales
singuli\`eres et op\'erateurs multilin\'eaires}, Ann. Inst. Fourier, Grenoble
\textbf{28} (1978), 177--202.

\bibitem{CM2} R. R.  Coifman and Y. Meyer,  {\it Au del\`a des op\'erateurs
pseudodiff\'erentiels}, Ast\'erisque  {\bf 57} (1978).

%\bibitem{CM3} R. R.  Coifman and Y. Meyer, {\it Ondelettes \'et op\'erateurs}, III, Hermann, Paris, 1990.


 \bibitem{DGY} X. T. Duong, L. Grafakos and L. X. Yan,
 {\it Multilinear   operators with non-smooth kernels and commutators of singular integrals},
  Trans. Amer. Math. Soc.  {\bf 362} (2010), 2089-2113.

\bibitem{FS} C. Fefferman and E. M. Stein, 
{\it Some maximal inequalities}, Amer. J. Math.  {\bf 93} (1971), 107--115.

  \bibitem{FuTo} M. Fujita and N. Tomita,
   \emph{Weighted norm inequalities for multilinear Fourier multipliers},
  Trans. Amer. Math. Soc.  \textbf{364} (2012),   6335--6353.

\bibitem{Grafakos1} L. Grafakos, {\it Classical  Fourier Analysis, 3rd Edition}, GTM 249, Springer, New York, 2014.

\bibitem{GrafakosMFA} L. Grafakos, {\it Modern  Fourier Analysis, 3rd Edition}, GTM 250, Springer, New York, 2014.



%\bibitem{GrHeHo} L. Grafakos, D. He,  and P. Honz\'\i k,
%\emph{Rough bilinear singular integrals}, submitted.

%\bibitem{GH} L. Grafakos and P. Honz\'\i k,
%\emph{ A weak type estimate for commutators},  Int. Math.
%Res. Not. IMRN {\bf 2012},   4785--4796.




\bibitem{GrNg}  L. Grafakos and H. V. Nguyen, {\it
Multilinear Fourier multipliers with minimal Sobolev regularity, I}, Colloq. Math., to appear.

\bibitem{GrMiNgTo}  L. Grafakos, A. Miyachi, H. V. Nguyen  and N. Tomita,  {\it
Multilinear Fourier multipliers with minimal Sobolev regularity, II}, submitted.

\bibitem{GrMiTo} L. Grafakos, A. Miyachi  and N. Tomita, \emph{On multilinear
Fourier multipliers of limited smoothness,}  Can. J.
Math.  \textbf{65} (2013),   299--330.

\bibitem{GrOh} L. Grafakos and S. Oh, \emph{The Kato-Ponce inequality},
 Comm. Partial Differential Equations  {\bf 39} (2014),  1128--1157.

\bibitem{GrSi} L. Grafakos and Z. Si, {\it  The H\"ormander multiplier theorem for multilinear operators},
 J. Reine Angew. Math.  {\bf 668}  (2012), 133--147.


 \bibitem {GT1} L. Grafakos and R. H. Torres,   On multilinear singular integrals of
 Calder\'on-Zygmund type.   Proceedings of the 6th International Conference on Harmonic
 Analysis and Partial Differential Equations (El Escorial, 2000).  {\it Publ. Mat.} (2002),   Extra, 57--91.

\bibitem {GT2} L. Grafakos and R. H. Torres,  {\it  Multilinear Calder\'on-Zygmund theory},
Adv.  Math.   {\bf 165} (2002), 124-164.

\bibitem{HLLW} Y. Han, J. Li, G. Lu, and P. Wang,
{\it $H^p\to H^p$ boundedness implies $H^p\to L^p$ boundedness},
Forum Math. {\bf  23} (2011),  729--756.

\bibitem{He2014}
D.~He, {\it   Square function characterization of weak Hardy spaces},
 J.   Fourier Anal.   Appl.   {\bf 20} (2014), 1083--1110.
 

\bibitem{Ho} L. H\"ormander,   \emph{Estimates for translation invariant operators in $L^p$ spaces},
Acta Math. \textbf{104}   (1960), 93--140.

 \bibitem{Jo1} J.-L. Journ\'e,
  {\it  Calder\'on--Zygmund operators on product spaces},
 Rev.~Mat.~Iberoamericana, {\bf 1} (1985), 55--91.

  \bibitem{Jo2} J.-L.  Journ\'e,   {\it Two problems of   Calder\'on--Zygmund operators on product spaces},
 Ann. Inst. Fourier, Grenoble, {\bf 38} (1988), 111--132.

  \bibitem{KP} T. Kato and G. Ponce,  {\it  Commutator estimates and the Euler and Navier-Stokes equations},
 Comm. Pure Appl. Math  {\bf XLI} (1988), 891--907.

\bibitem{KS} C. E. Kenig and E. M. Stein,  {\it Multilinear estimates and fractional
integration},
 Math. Res. Letters, {\bf 6} (1999), 1--15.

\bibitem{Mi}  S. G. Mikhlin,    \emph{On the multipliers of Fourier integrals}, (Russian) Dokl. Akad. Nauk
SSSR (N.S.) \textbf{109} (1956), 701--703.

\bibitem{MiTo} A. Miyachi   and N. Tomita, \emph{Minimal smoothness conditions for bilinear
 Fourier multipliers,} Rev. Mat. Iberoamericana  \textbf{29} (2013),   495--530.
 
 \bibitem{MiTo14}  A. Miyachi   and N. Tomita, \emph{Boundedness criterion for bilinear 
 Fourier multiplier operators},  Tohoku Math. J. (2) \textbf{66} (2014), 55--76.

 \bibitem{Mu1} C. Muscalu,  \emph{Calder\'on commutators and the Cauchy integral
 on Lipschitz curves revisited I. First commutator and generalizations},
 Rev. Mat. Iberoam. {\bf 30} (2014),   727--750.

%\bibitem{Mu2} C. Muscalu,  \emph{Calder\'on commutators and the Cauchy integral
%on Lipschitz curves revisited II. The Cauchy
%integral and its generalizations},
%Rev. Mat. Iberoam. {\bf 30} (2014),  1089--1122.

 \bibitem{Mu3} C. Muscalu,  \emph{Calder\'on commutators and the Cauchy integral
 on Lipschitz curves revisited III. Polydisc extensions},
 Rev. Mat. Iberoam. {\bf 30} (2014),  1413--1437.

\bibitem{MuPiTaTh} C. Muscalu, J. Pipher, T. Tao  and C. Thiele, {\it Bi-parameter paraproducts},
Acta Math. {\bf 193} (2004), 269--296.

\bibitem{MuPiTaTh2} C. Muscalu, J. Pipher, T. Tao  and C. Thiele, {\it Multi-parameter paraproducts},
Rev. Mat. Iberoamericana {\bf 22} (2006),    963--976.

%\bibitem{ASeegerRev} A. Seeger,
%\emph{A weak type bound for a singular integral},
%Rev. Mat. Iberoam. {\bf 30} (2014),   961--978.

%\bibitem{SSS} A. Seeger, C. K. Smart, and B. Street, {\it
%Multilinear singular integral forms of Christ-Journ\'e type}, preprint http://arxiv.org/abs/1510.06990
 \bibitem{STr} H. H.~Schmeisser, H.~Triebel,
{\em Topics in Fourier analysis and function spaces},
Mathematik und ihre Anwendungen in Physik und Technik [Mathematics and its Applications in Physics and Technology], {\bf 42}. Akademische Verlagsgesellschaft Geest \& Portig K.-G., Leipzig, 1987. 300 pp. 

\bibitem{St2} E. M. Stein,  {\it The characterization of functions arising as potentials},
Bull. Amer. Math. Soc.
{\bf 67}  (1961), 102--104.


 \bibitem{St1} E. M. Stein,  {\it Singular integral and
 differentiability properties of  functions},   Princeton Univ. Press,
 Princeton, NJ, 1970.

%\bibitem{St} E. M. Stein, {\it Harmonic Analysis: Real variable
%  methods, orthogonality and oscillatory integrals,}   Princeton Univ. Press, Princeton, NJ,  1993.

 \bibitem{To}  N.  Tomita,
 \emph{A H\"ormander type multiplier theorem for multilinear operators,} J. Funct. Anal.
\textbf{259} (2010), 2028--2044.

  \bibitem{Tri} H. Triebel,
  {\it Theory of function spaces  II}, Monographs in Mathematics,
  {\bf 84}  Birkh\"auser Verlag, Basel, 1992. viii+370 pp.

 \end{thebibliography}
 \end{document}